\documentclass[12pt,oneside,a4paper]{amsart}
\usepackage[utf8]{inputenc}
\usepackage[T1]{fontenc}
\usepackage{amssymb}
\usepackage{enumitem}
\usepackage{thmtools}
\usepackage[hidelinks]{hyperref}
\usepackage{mathtools}

\makeatletter
\def\@tocline#1#2#3#4#5#6#7{\relax
  \ifnum #1>\c@tocdepth 
  \else
    \par \addpenalty\@secpenalty\addvspace{#2}%
    \begingroup \hyphenpenalty\@M
    \@ifempty{#4}{%
      \@tempdima\csname r@tocindent\number#1\endcsname\relax
    }{%
      \@tempdima#4\relax
    }%
    \parindent\z@ \leftskip#3\relax \advance\leftskip\@tempdima\relax
    \rightskip\@pnumwidth plus4em \parfillskip-\@pnumwidth
    #5\leavevmode\hskip-\@tempdima
      \ifcase #1
       \or\or \hskip 1em \or \hskip 2em \else \hskip 3em \fi%
      #6\nobreak\relax
    \hfill\hbox to\@pnumwidth{\@tocpagenum{#7}}\par
    \nobreak
    \endgroup
  \fi}
\makeatother

\declaretheorem[numberwithin=section]{theorem}
\declaretheorem[numberlike=theorem]{lemma}

\declaretheorem[numberlike=theorem]{proposition}
\declaretheorem[numberlike=theorem]{algorithm}
\declaretheorem[style=definition,numberlike=theorem]{definition}
\declaretheorem[style=definition,numberlike=theorem]{example}
\declaretheorem[style=remark,numberlike=theorem]{remark}

\newcommand{\terminology}[1]{\emph{#1}}
\newcommand{\abs}[1]{\left\lvert#1\right\rvert}
\newcommand{\restr}[2]{\left.#1\right\vert_{#2}}
\newcommand{\der}{\operatorname{der}}
\newcommand{\Aut}{\operatorname{Aut}}
\newcommand{\ad}{\operatorname{ad}}
\newcommand{\Ad}{\operatorname{Ad}}
\newcommand{\positiveset}{\mathbf{A}_+}
\newcommand{\norm}[1]{\left\lVert #1 \right\rVert}
\newcommand{\braket}[1]{\left\langle #1 \right\rangle}
\newcommand{\lt}{<}
\newcommand{\gt}{>}
\newcommand{\amp}{&}
\title{Gradings for nilpotent Lie algebras}
\author{Eero Hakavuori}
\address[Hakavuori]{SISSA}
\email{eero.hakavuori@sissa.it}
\author{Ville Kivioja}
\address[Kivioja]{University of Jyväskylä}
\email{kivioja.ville@gmail.com}
\author{Terhi Moisala}
\address[Moisala]{University of Jyväskylä}
\email{moisala.terhi@gmail.com}
\author{Francesca Tripaldi}
\address[Tripaldi]{University of Bern}
\email{francesca.tripaldi@math.unibe.ch}
\date{January 11, 2022}
\begin{document}
\label{gradings}

\begin{abstract}
We present a constructive approach to torsion-free gradings of Lie algebras. Our main result is the computation of a maximal grading. Given a Lie algebra, using its maximal grading we enumerate all of its torsion-free gradings as well as its positive gradings. As applications, we classify gradings in low dimension, we consider the enumeration of Heintze groups, and we give methods to find bounds for non-vanishing \(\ell^{q,p}\) cohomology.
\end{abstract}


\maketitle
\tableofcontents\typeout{************************************************}
\typeout{Section 1 Introduction}
\typeout{************************************************}

\section{Introduction}\label{section-introduction}

\typeout{************************************************}
\typeout{Subsection 1.1 Overview}
\typeout{************************************************}

\subsection{Overview}\label{g:subsection:idp1}
A \terminology{grading} of a Lie algebra \(\mathfrak{g}\) is a direct sum decomposition
\begin{equation}
\mathfrak{g}=\bigoplus_{\alpha\in A}V_\alpha\label{eq-grading-direct-sum}
\end{equation}
indexed by an abelian group \(A\) in such a way that
\begin{equation}
[V_\alpha,V_\beta]\subset V_{\alpha+\beta}\label{eq-grading-relations}
\end{equation}
for each pair \(\alpha,\beta\in A\). In this paper, we will focus on Lie algebras over fields of characteristic zero and torsion-free gradings, i.e., gradings indexed over torsion-free abelian groups. There are also more general notions of gradings of Lie algebras or more general algebras, see for instance \cite{Patera-Zassenhaus-1989_on_lie_gradings_i} or \cite{Cornulier-2016-gradings}.

An important example of a Lie algebra grading is the so called \terminology{maximal grading}, which has the property of not admitting any proper refinement into smaller subspaces \(V_\alpha\). Slightly different notions of a maximal grading (sometimes also \terminology{fine grading}) exist in the literature. In this paper, we follow a similar viewpoint as in \cite{Favre-1973-system_de_poids}, and define maximal gradings by using maximal split tori of the derivation algebra \(\der(\mathfrak g)\), see {Definition~{\ref{def-maximal-grading}}}. The subspaces \(V_\alpha\) of a maximal grading are then the maximal subspaces of \(\mathfrak{g}\) where every derivation of the torus acts by scaling. In addition, the indexing group of a maximal grading is some \(\mathbb{Z}^k\) with the universal property that every other torsion-free grading can be obtained using projections, see {Proposition~{\ref{prop-max-grading}}}.

A classical example of a maximal grading is the Cartan decomposition, which plays a fundamental role in representation theory and the classification of semisimple Lie algebras over \(\mathbb{C}\), see for example \cite{Humphreys-1978-representation_theory}. There has been a growing interest in the study of (maximal) gradings of semisimple Lie algebras since the paper \cite{Patera-Zassenhaus-1989_on_lie_gradings_i}, see the survey \cite{Kochetov-2009-gradings_on_simple_lie_algebras_survey} or the monograph \cite{Elduque-Kochetov-2013-gradings_on_simple_lie_algebras} for an overview. Moreover, a classification of maximal gradings of simple classical Lie algebras over algebraically closed fields of characteristic zero can be found in \cite{Elduque-2010-fine_gradings_on_simple_lie_algebras}.

Regarding nilpotent Lie algebras over algebraically closed fields of characteristic zero, an in depth study of maximal gradings over torsion-free abelian groups was carried out in \cite{Favre-1973-system_de_poids}. One of the main results in \cite{Favre-1973-system_de_poids} is that within the family of nilpotent Lie algebras \(\mathfrak{g}\) of nilpotency step \(s\) and with abelianization \(\mathfrak{g}/[\mathfrak{g},\mathfrak{g}]\) of dimension \(r\), there are only finitely many torsion-free maximal gradings, up to automorphisms of the free nilpotent Lie algebra of step \(s\) and rank \(r\). This finiteness in the number of maximal gradings is in contrast with for example the existence of an uncountable number of non-isomorphic nilpotent Lie algebras over the reals in dimension 7 and higher.

There are two other special types of gradings of particular interest in the case of nilpotent Lie algebras: \terminology{positive gradings} and \terminology{stratifications}. A positive grading is a grading indexed over the reals such that in the direct sum decomposition {({\ref{eq-grading-direct-sum}})} all the non-zero spaces \(V_\alpha\) have positive indices \(\alpha\gt 0\). A stratification (also called a \terminology{Carnot grading}) is a positive grading for which \(V_1\) generates \(\mathfrak{g}\) as a Lie algebra.

Lie algebras with a stratification are the Lie algebras of Carnot groups. These groups play a central role in the fields of geometric analysis, geometric measure theory, and large scale geometry, see \cite{Le_Donne-2017-carnot_primer} for a list of references.

Positive gradings are important within the study of homogeneous spaces, as they appear directly in characterizations of such spaces. First, any negatively curved homogeneous Riemannian manifold is a \terminology{Heintze group} \(G\rtimes \mathbb{R}\) \cite{Heintze-1974-homogeneous_manifolds_of_negative_curvature}, where \(G\) is a nilpotent Lie group and the action of \(\mathbb{R}\) on \(G\) is given by a one-parameter family of automorphisms inducing a positive grading of \(G\). Second, any connected locally compact group that admits a contracting automorphism is a positively gradable Lie group \cite{Siebert-1986-contractive_automorphisms}. In this latter result, the group structure and contracting automorphism may also be replaced by a metric structure and a dilation, see \cite{CKLNO-2017-homogeneous_metric_spaces_to_lie_groups}.

Another active area of research that contains several open problems related to positively gradable Lie groups is the quasi-isometric classification of locally compact groups. A survey on the topic can be found in \cite{quasisurvey}. For instance, it is not known whether there exists a non-stratifiable positively gradable Lie group that is quasi-isometric to its asymptotic cone, nor whether all large-scale contractible groups are positively gradable, see \cite[Question~7.9]{Cornulier-2019-sublinear_bilipschitz_equivalence}. The quasi-isometric classification is open also for Heintze groups, see \cite{Carrasco-Sequeira-2017-qi_invariants_associated_to_a_heintze_group} for some known results.

\typeout{************************************************}
\typeout{Subsection 1.2 Implemented algorithms and applications}
\typeout{************************************************}

\subsection{Implemented algorithms and applications}\label{g:subsection:idp2}
The purpose of our work is to implement algorithms for computing various gradings of a Lie algebra. In this paper we give a thorough explanation of the algorithms together with relevant mathematical concepts and proofs of correctness, as well as present some applications. The source code for the implementation can be found in \cite{software-lie_algebra_gradings}, along with a full classification of gradings in dimension 6 and for ``most'' Lie algebras in dimension 7, in the sense explained in {Subsection~{\ref{applications-classification}}}.

\label{par-structure-coefficients}
In all of the algorithms, we work with a finite dimensional Lie algebra \(\mathfrak{g}\) defined by its structure coefficients in some field \(F\), which does not have to be the base field of the Lie algebra and will usually be smaller. That is, we assume we have a fixed basis \(X_1,\ldots,X_n\) of \(\mathfrak{g}\) and are given a family of coefficients \(\{c_{ij}^k\in F: i,j,k\in\{1,\ldots,n\}\}\) such that the Lie bracket is defined as
\begin{equation*}
[X_i,X_j] = \sum_{k=1}^nc_{ij}^kX_k\text{.}
\end{equation*}
In practice, \(F\) will usually be either the field of rationals or algebraic numbers. The actual base field of the Lie algebra will usually not be relevant, see the discussion in {Subsection~{\ref{ssec-base-fields}}}.

The most important of the implemented algorithms is the construction of a {maximal grading} of a Lie algebra \(\mathfrak{g}\) when the field \(F\) is algebraically closed, see {Algorithm~{\ref{algorithm-maximal-grading}}}.

Given a {maximal grading} of a Lie algebra \(\mathfrak{g}\), it is straightforward to construct a finite collection of gradings, which contains the universal realizations (see {Definition~{\ref{def-universal-realization}}}) of all torsion-free gradings of \(\mathfrak{g}\) up to equivalence in the sense of {Definition~{\ref{def-equivalent-group-gradings}}}. We cover this enumeration procedure in {Subsection~{\ref{subsection-torsion-free}}}. The finite collection will in general contain equivalent gradings. In the classification of gradings given in \cite{software-lie_algebra_gradings} we eliminate this redundancy in the case of nilpotent Lie algebras of dimension up to 6.

A maximal grading of a Lie algebra \(\mathfrak{g}\) also determines a parametrization of all the positive gradings of \(\mathfrak{g}\) as a convex cone. We explain this in detail in {Subsection~{\ref{sssec-pos-gradings}}} and present {Algorithm~{\ref{algorithm-positive-realization}}} which is a method to construct a positive realization of a grading when one exists. We also give a construction for a {stratification} (when one exists) in {Subsection~{\ref{section-stratifications}}}. The stratifiability criterion is the same as the one in \cite[Lemma~3.10]{Cornulier-2016-gradings}, which we write out as an explicit linear system.

Finally we give two applications of the parametrization of positive gradings. In {Proposition~{\ref{prop-noneqgradings-produce-diff-heintze}}} we show that all non-equivalent positive gradings define non-isomorphic Heintze groups. Thus we give methods to enumerate diagonal Heintze groups. Consequently, we make progress on the problem of finding all Heintze groups with prescribed nilradical, which is a question already hinted by Heintze, see the discussion after Theorem~3 in \cite{Heintze-1974-homogeneous_manifolds_of_negative_curvature}.

As another application, the parametrization of positive gradings gives a method to find better estimates for the non-vanishing of the \(\ell^{q,p}\) cohomology of a nilpotent Lie group, which is a quasi-isometry invariant.

\typeout{************************************************}
\typeout{Subsection 1.3 Structure of the paper}
\typeout{************************************************}

\subsection{Structure of the paper}\label{g:subsection:idp3}
In {Section~{\ref{section-preliminaries}}} we recall various definitions and terminology. In {Subsection~{\ref{ssec-base-fields}}} we discuss to what extent we need to worry about the base fields of the Lie algebras appearing throughout the paper. The core concepts of {realization}, {push-forward}, and {equivalence} are defined in {Subsection~{\ref{ssec-gradings-and-equivalences}}} and {universal realizations} are recalled in {Subsection~{\ref{ssec-univ-gradings}}}. {Subsection~{\ref{sec-gradings-by-tori}}} covers how to study {torsion-free} gradings of a Lie algebra \(\mathfrak{g}\) in terms of tori of the derivation algebra \(\der(\mathfrak{g})\). {Maximal gradings} and their {universal property} are discussed in {Subsection~{\ref{section-maximal-gradings}}}.

In {Section~{\ref{section-constructions}}} we give our main constructions. {Subsection~{\ref{subsection-torsion-free}}} reduces the enumeration problem of all torsion-free gradings to the construction of a maximal grading. {Subsection~{\ref{section-stratifications}}} covers the construction of a {stratification}. {Subsection~{\ref{sssec-pos-gradings}}} considers {Algorithm~{\ref{algorithm-positive-realization}}} on {positive gradings}. The construction of a maximal grading, {Algorithm~{\ref{algorithm-maximal-grading}}}, is described in {Subsection~{\ref{section-maximal-construction}}}.

In {Section~{\ref{section-applications}}} we give various applications of gradings to the study of Lie algebras and Lie groups. {Subsection~{\ref{applications-structure}}} shows how to use the maximal grading of a Lie algebra as a tool to detect decomposability of a Lie algebra, and how to reduce the dimension of the problem of deciding whether two Lie algebras are isomorphic. In {Subsection~{\ref{applications-classification}}} we classify up to equivalence the gradings of low dimensional nilpotent Lie algebras over \(\mathbb{C}\). In {Subsection~{\ref{applications-heintze}}} we cover the results on the enumeration of {Heintze groups}. Finally, in {Subsection~{\ref{applications-lpq}}}, we present the method for finding improved bounds for the non-vanishing of the {\(\ell^{q,p}\) cohomology}.

\typeout{************************************************}
\typeout{Section 2 Preliminaries}
\typeout{************************************************}

\section{Preliminaries}\label{section-preliminaries}
The contents of this section can, up to some modifications, be found in \cite[Section~3-4]{Kochetov-2009-gradings_on_simple_lie_algebras_survey}. Nonetheless, we give here a self-contained presentation to better fit our constructive approach.

\typeout{************************************************}
\typeout{Subsection 2.1 On base fields}
\typeout{************************************************}

\subsection{On base fields}\label{ssec-base-fields}
Throughout the paper, we will always take for granted that the base fields for all Lie algebras and other vector spaces are of characteristic zero. This characteristic zero assumption is essential in order to work interchangeably with gradings of a Lie algebra \(\mathfrak{g}\) and tori of the derivation algebra \(\der(\mathfrak{g})\), see \cite[Section~1.4]{Elduque-Kochetov-2013-gradings_on_simple_lie_algebras} for some discussion on the non-zero characteristic case.

At several points we will ignore the base field \(F'\) of the Lie algebra and only work with the potentially smaller field \(F\subset F'\) that the Lie algebra is defined over.
\begin{definition}\label{def-defined-over}
A Lie algebra \(\mathfrak{g}\) with base field \(F'\) is \terminology{defined over} a field \(F\subset F'\) if it has a basis such that the structure coefficients in that basis are all elements of \(F\).
\end{definition}
All the constructions given in {Section~{\ref{section-constructions}}} are in fact independent of the actual base field \(F'\supset F\). In the case where the construction does not involve the base field at all this independence is automatic, such as in the enumeration of gradings in {Subsection~{\ref{subsection-torsion-free}}}. Otherwise we will mention the reason explicitly either by a reference such as for stratifications in {Remark~{\ref{remark-stratification-field-extension}}} or by an immediate argument as for maximal gradings in {Remark~{\ref{remark-maximal-grading-field-extension}}}.

The independence of the base field is particularly convenient because it allows us to make use of computer algebra systems, where it is necessary to work with a \terminology{computable} field, i.e.\@, it is necessary to be able to distinguish elements and compute all the field operations with finite algorithms. For instance, although the field of reals \(\mathbb{R}\) is not computable, we may make use of our computer algebra implementation as soon as the Lie algebra is defined over the rationals or the real algebraic numbers, i.e.\@, the real numbers that are algebraic over the rationals. As an example, in {Subsection~{\ref{applications-classification}}} we will give a brief overview of our classification of gradings of all Lie algebras up to dimension 6 over \(\mathbb{C}\).

\typeout{************************************************}
\typeout{Subsection 2.2 Gradings and equivalences}
\typeout{************************************************}

\subsection{Gradings and equivalences}\label{ssec-gradings-and-equivalences}
In this section we define some key notions related to gradings of Lie algebras, including {push-forward} and {equivalence}.
\begin{definition}\label{def-grading}
\label{g:notation:idp4}
\label{g:notation:idp5}
\label{g:notation:idp6}
A \terminology{grading} of a Lie algebra \(\mathfrak{g}\) is a direct sum decomposition \(\mathcal V : \mathfrak{g}=\bigoplus_{\alpha\in A}V_\alpha\), where \(A\) is an abelian group and for each \(\alpha,\beta\in A\) it holds \([V_\alpha,V_\beta]\subset V_{\alpha+\beta}\). The group \(A\) is called the \terminology{grading group} of the grading \(\mathcal{V}\), and we say that the grading \(\mathcal{V}\) is \terminology{over \(A\)}, or that \(\mathcal{V}\) is an \terminology{\(A\)-grading}.

The subspaces \(V_\alpha\) are called the \terminology{layers} of the grading \(\mathcal{V}\). Each layer \(V_\alpha\) is also commonly referred to as the \terminology{homogeneous component of degree \(\alpha\)}. The elements \(\alpha \in A\) such that \(V_\alpha\neq 0\) are called the \terminology{weights} of \(\mathcal V\). We define the \terminology{support} of a grading as its set of weights, and we will denote it by \(\Omega\). A basis of \(\mathfrak g\) is said to be \terminology{adapted to \(\mathcal V\)}, or alternatively is said to be a \terminology{homogeneous basis}, if every element of the basis is contained in some layer of \(\mathcal V\).
\end{definition}
\begin{definition}\label{def-torsion-free-grading}
A {grading} is called \terminology{torsion-free} if its {grading group} is torsion-free.
\end{definition}
\begin{definition}\label{def-push-forward-grading}
\label{g:notation:idp7}
\label{g:notation:idp8}
\label{g:notation:idp9}
Let \(\mathcal V : \mathfrak g = \bigoplus_{\alpha \in A}V_\alpha \) be a {grading} over some abelian group \(A\). Given an automorphism \(\Phi \in \Aut(\mathfrak g )\), an abelian group \(B \) and a homomorphism \(f \colon A \to B \), we define the \terminology{push-forward grading} \(f_* \Phi(\mathcal V) : \mathfrak g =  \bigoplus_{\beta \in B}W_\beta \) over \(B\), where
\begin{equation*}
W_\beta= \bigoplus_{\alpha \in f^{-1}(\beta)} \Phi(V_{\alpha}).
\end{equation*}
When \(\Phi=\operatorname{Id}\), we simply denote \(f_* \operatorname{Id} (\mathcal{V}) = f_* \mathcal{V}\).
\end{definition}
It is readily checked that the push-forward grading is indeed a \(B\)-grading in the sense of {Definition~{\ref{def-grading}}}.

There are several different notions of equivalence of gradings in the literature. The one that we shall use is called \terminology{group-equivalence} in \cite{Kochetov-2009-gradings_on_simple_lie_algebras_survey}. For brevity, we will refer to this notion as equivalence. Stated in terms of push-forwards, the group-equivalence notion of \cite{Kochetov-2009-gradings_on_simple_lie_algebras_survey} takes the following form:
\begin{definition}\label{def-equivalent-group-gradings}
An {\(A\)-grading} \(\mathcal V\) and a \(B\)-grading \(\mathcal W\) are said to be \terminology{equivalent} if there exist an automorphism \(\Phi \in \Aut(\mathfrak g)\) and a group isomorphism \(f \colon A \to B \) such that \(\mathcal W = f_* \Phi(\mathcal V)\).
\end{definition}
The equivalence of an \(A\)-grading and a \(B\)-grading has a well known characterization in terms of push-forwards. We provide a brief proof for completeness.
\begin{lemma}\label{lemma-pf-with-generating-weights}
Let \(\mathcal V : \mathfrak g = \bigoplus_{\alpha \in A}V_\alpha \) and \(\mathcal W : \mathfrak g = \bigoplus_{\beta \in B}W_\beta \) be {gradings} such that the {weights} of \(\mathcal V \) and \(\mathcal W\) generate the abelian groups \(A\) and \(B\), respectively. If there exist homomorphisms \(f \colon A \to B \) and \(g\colon B \to A\) such that \(\mathcal W = f_*\mathcal V\) and \(\mathcal V = g_*\mathcal W\), then \(\mathcal V\) and \(\mathcal W\) are {equivalent}.
\end{lemma}
\begin{proof}\label{g:proof:idp10}
Let us denote by \(\Omega_A\) and \(\Omega_B\) the sets of weights of \(\mathcal V\) and \(\mathcal W\). Notice first that by definition of the {push-forward}, \(f(\Omega_A) = \Omega_B\) and \(g(\Omega_B) = \Omega_A\), so both \(f\) and \(g\) are injective on weights. Moreover, we have for every \(\alpha \in \Omega_A\) and \(\beta \in \Omega_B\) the correspondence
\begin{equation*}
V_\alpha = W_{f(\alpha)} = V_{g(f(\alpha))} \qquad \text{and} \qquad W_\beta = V_{g(\beta)} = W_{f(g(\beta))}.
\end{equation*}
Hence \(f\colon \Omega_A \to \Omega_B \) is a bijection and \(f^{-1} = g\) on \(\Omega_B\). Since \(\Omega_A\) and \(\Omega_B\) generate \(A \) and \(B\) as groups, we get that \(f^{-1} = g\) on whole \(B\).
\end{proof}
Notice that the assumption on generating weights is indeed necessary: for instance, the gradings \(\mathbb{R} = V_1 \) over \(\mathbb{Z}\) and \(\mathbb{R} = V_{(1,0)} \) over \(\mathbb{Z}^2\) are push-forward gradings of each other, but they are not equivalent.

\typeout{************************************************}
\typeout{Subsection 2.3 Universal realizations}
\typeout{************************************************}

\subsection{Universal realizations}\label{ssec-univ-gradings}
In some cases the grading in hand plays a role of a partition of the Lie algebra, where the indexing of the layers is unnatural or irrelevant, see for instance \cite{Patera-Zassenhaus-1989_on_lie_gradings_i} and \cite[Section~1.1]{Elduque-Kochetov-2013-gradings_on_simple_lie_algebras}. One may want to equip such a grading with a new labeling, which has more convenient algebraic structure or which reveals some special properties of the grading, like positivity. Such a reindexing is said to be another \terminology{realization} of the original grading.
\begin{definition}\label{def-realization}
Let \(\mathcal V : \mathfrak{g}=\bigoplus_{\alpha\in A}V_\alpha\) be a grading with {weights} \(\Omega \subset A \). Suppose we have an embedding \(f \colon \Omega \to B\) into an abelian group \(B\) such that \(\mathcal{W}: \mathfrak{g}=\bigoplus_{\beta\in B}W_\beta\) is a grading, where \(W_{f(\alpha)}=V_\alpha\). Then the resulting \(B\)-grading \(\mathcal{W}\) is called a \terminology{realization} of the grading \(\mathcal V\).
\end{definition}
We do not in general require that the {weights of an \(A\)-grading} generate the grading group \(A\) in order to include e.g. gradings over \(A=\mathbb{R}\) in the discussion. Moreover, weights of a grading may have additional relations coming from the ambient group structure, even when the corresponding layers are unrelated. To build a satisfactory theory using homomorphisms between grading groups, we consider the notion of an (abelian) {universal realization}, see \cite[Section~3.3]{Kochetov-2009-gradings_on_simple_lie_algebras_survey}.
\begin{definition}\label{def-universal-realization}
Let \(\mathcal V \) be a {grading} of \(\mathfrak g\). A \terminology{universal realization} of \(\mathcal V\) is a {realization}  \(\widetilde {\mathcal V}\) as an \(A\)-grading such that for every realization of \(\mathcal{V}\) as a \(B\)-grading, there exists a unique homomorphism \(f \colon A \to B\) such that the \(B\)-grading is the {push-forward grading}  \(f_*\widetilde {\mathcal V}\).
\end{definition}
Observe that by {Lemma~{\ref{lemma-pf-with-generating-weights}}}, the universal realization of a grading is unique up to equivalence.

\label{par-span-brackets}
The universal realization of a grading \(\mathcal{V}\) can be constructed by considering the free abelian group generated by the weights and quotienting out the grading relations as described below. In this paper, the notation \(\langle X\rangle\) always refers to the span of \(X\) in the appropriate sense. \label{g:notation:idp11}
\begin{enumerate}
\item\label{alg-grading-relations}Enumerate the support of \(\mathcal V\) as \(\Omega = \{\alpha_1,\dots,\alpha_n \}\) and let \(\{e_1,\dots,e_n\}\) be the canonical basis of \(\mathbb{Z}^n\).
\item\label{alg-grading-relations2}Enumerate the grading relations \(R\subset\mathbb{Z}^n\) as follows. For each pair \(\alpha_i, \alpha_j \in \Omega\) such that \([V_{\alpha_i}, V_{\alpha_j}] \neq 0\), add the element \(e_i + e_j - e_k\) to \(R\), where \(k\) is the index such that \(\alpha_k = \alpha_i+\alpha_j\).
\item\label{alg-univ-real-quotient}For all \(i=1,\dots,n\), set \(\widetilde V_{\pi(e_i)} = V_{\alpha_i}\), where \(\pi \colon \mathbb{Z}^n \to \mathbb{Z}^n/\langle R\rangle\) is the projection. The resulting \(\mathbb{Z}^n/\langle R\rangle\)-grading \(\widetilde {\mathcal V}\) is the universal realization.
\end{enumerate}

To see that the grading \(\widetilde {\mathcal V}\) is well-defined, consider the homomorphism \(\phi \colon \mathbb{Z}^n \to A\) defined by \(\phi(e_i)=\alpha_i\) for all \(1 \le i \le n\), where \(A\) is the grading group of \(\mathcal{V}\). If \(\pi(e_i)=\pi(e_j)\), then \(e_i-e_j \in \langle R\rangle\) and we have \(\alpha_i=\phi(e_i)=\phi(e_j)=\alpha_j\) since \(R\subset\ker\phi\). Moreover, the obtained \(\mathbb{Z}^n/\langle R\rangle\)-grading is a universal realization of \(\mathcal V\) by the universal property of quotients, since the construction of \(\widetilde{\mathcal{V}}\) does not depend on the realization of \(\mathcal{V}\) that we start with.

In the rest of the paper we will focus on gradings that admit torsion-free realizations. For such gradings, the universal realizations are gradings over some \(\mathbb{Z}^k\), as demonstrated by the following lemma.
\begin{lemma}\label{lemma-universal-realization-Zk}
If \(\mathcal V\) is a {torsion-free grading}, then the {grading group}  of the {universal realization}  of \(\mathcal V\) is isomorphic to some \(\mathbb{Z}^k\).
\end{lemma}
\begin{proof}\label{g:proof:idp12}
Let \(\widetilde{\mathcal{V}}\) be the  universal realization of \(\mathcal{V}\), so \(\widetilde{\mathcal{V}}\) is a \(\mathbb{Z}^n/\langle R\rangle\)-grading for some subset \(R\subset \mathbb{Z}^n\). The quotient \(\mathbb{Z}^n/\langle R\rangle\) is isomorphic to a group \(\mathbb{Z}^k \times G_t\), where \(G_t\) is some torsion group.

By assumption the grading group \(A\) of \(\mathcal{V}\) is torsion-free. Since the image of \(G_t\) under a homomorphism must vanish in \(A\), we conclude that there are no non-zero {weights} in \(G_t\). Since the grading group of a universal realization is generated by the weights, we conclude that \(G_t=0\), and \(\widetilde{\mathcal{V}}\) is a \(\mathbb{Z}^k\)-grading.
\end{proof}
The following lemma is a part of \cite[Proposition~3.15]{Kochetov-2009-gradings_on_simple_lie_algebras_survey}, and we record it for later usage.
\begin{lemma}\label{lemma-universal-of-coarsening}
If a {grading} \(\mathcal V\) is a {push-forward} of a grading \(\mathcal W \), then every {realization} of \(\mathcal V\) is a {push-forward grading} of the {universal realization} of \(\mathcal W\).
\end{lemma}

\typeout{************************************************}
\typeout{Subsection 2.4 Gradings induced by tori}
\typeout{************************************************}

\subsection{Gradings induced by tori}\label{sec-gradings-by-tori}
\label{par-tori-intro}
\label{g:notation:idp13} In this subsection we describe the correspondence between {gradings} of a Lie algebra \(\mathfrak{g}\) and the {split tori} of its derivation algebra \(\der(\mathfrak{g})\). In general, gradings of a Lie algebra \(\mathfrak{g}\) are in one-to-one correspondence with algebraic quasitori, see \cite[Section~4]{Kochetov-2009-gradings_on_simple_lie_algebras_survey}. However, in this study we are only interested in cases when \(\mathfrak{g}\) is a finite-dimensional Lie algebra over a field of characteristic zero and the {gradings are over torsion-free abelian groups}. In this setting, the characterization of gradings in terms of algebraic quasitori can be reduced to studying algebraic subtori of the derivation algebra \(\der(\mathfrak{g})\).

For computational reasons, we will drop the algebraicity requirement for the subalgebras of \(\der(\mathfrak{g})\). This means we lose the one-to-one correspondence described in \cite{Kochetov-2009-gradings_on_simple_lie_algebras_survey}, but the less restrictive definition will be sufficient for our purposes. In particular, it will simplify the explicit construction of {maximal gradings} in terms of tori, see {Subsection~{\ref{section-maximal-construction}}}.

We start by defining split tori and {gradings induced} by them in the sense of \cite{Favre-1973-system_de_poids}.
\begin{definition}\label{def-torus}
An abelian subalgebra \(\mathfrak{t}\) of semisimple derivations of \(\mathfrak{g}\) is called a \terminology{torus} of \(\der(\mathfrak g)\). If the torus \(\mathfrak{t}\) is diagonalizable over the base field of \(\mathfrak{g}\), it is called a \terminology{split torus}.
\end{definition}
\begin{lemma}\label{lemma-gradings-induced}
Let \(\mathfrak t\) be a {split torus} of \(\der(\mathfrak g)\) and let \(\mathfrak{t}^*\) be its dual as a vector space. For each \(\alpha\in\mathfrak{t}^*\) define the subspace
\begin{equation*}
V_\alpha=\{X\in\mathfrak{g}: \delta(X)=\alpha(\delta)X\,\forall \delta\in\mathfrak{t}\}\text{.}
\end{equation*}
Then \(\mathfrak{g}=\bigoplus_{\alpha\in\mathfrak{t}^*}V_\alpha\) is a {\(\mathfrak{t}^*\)-grading}.
\end{lemma}
\begin{proof}\label{g:proof:idp14}
Let \(X_1,\ldots,X_n\) be a basis of \(\mathfrak{g}\) that diagonalizes \(\mathfrak{t}\). Since each vector \(X_i\) is an eigenvector of every derivation \(\delta\in\mathfrak{t}\), there are well defined linear maps \(\alpha_1,\ldots,\alpha_n\in\mathfrak{t}^*\) determined by
\begin{equation*}
\delta(X_i) = \alpha_i(\delta)X_i,\quad i=1,\ldots,n\text{.}
\end{equation*}
By construction \(X_i\in V_{\alpha_i}\), so the direct sum \(\bigoplus_{\alpha\in\mathfrak{t}^*}V_\alpha\) spans all of the Lie algebra \(\mathfrak{g}\). The inclusion \([V_\alpha,V_\beta]\subset V_{\alpha+\beta}\) follows by linearity from the Leibniz rule \(\delta([X,Y]) = [\delta(X),Y]+[X,\delta(Y)]\) for all derivations \(\delta\in\mathfrak{t}\) and vectors \(X\in V_\alpha\) and \(Y\in V_\beta\).
\end{proof}
\begin{definition}\label{def-grading-induced-by-torus}
The {\(\mathfrak{t}^*\)-grading} of \(\mathfrak{g}\) defined in {Lemma~{\ref{lemma-gradings-induced}}} is called the \terminology{grading induced by the split torus \(\mathfrak{t}\)}.
\end{definition}
For the purposes of {Subsection~{\ref{section-maximal-gradings}}}, we need the following two lemmas. In {Lemma~{\ref{lemma-conjugates-and-subtori}}} we link {equivalences} and {push-forwards} of gradings to relations between the inducing tori.
\begin{lemma}\label{lemma-conjugates-and-subtori}
Let \(\mathfrak{t}_1\) and \(\mathfrak{t}_2\) be two {split tori} of \(\der(\mathfrak{g})\) with respective {induced} {\(\mathfrak{t}_1^*\)-grading} \(\mathcal V\) and {\(\mathfrak{t}_2^*\)-grading} \(\mathcal W\).
\begin{enumerate}[label=(\roman*)]
\item\label{enum-conj}If there exists an automorphism \(\Phi \in \Aut(\mathfrak{g})\) such that \(\Phi \circ \mathfrak t_1 \circ \Phi^{-1} = \mathfrak t_2\), then \(\mathcal V\) and \(\mathcal W\) are {equivalent}.
\item\label{enum-subtori}If \(\mathfrak t_1 \subset \mathfrak t_2\), then there exists a homomorphism \(f\) such that {\(\mathcal V = f_* \mathcal W\)}.
\end{enumerate}

\end{lemma}
\begin{proof}\label{g:proof:idp15}
To show \ref{enum-conj}, suppose that \(\Ad_{\Phi}\mathfrak{t}_1=\Phi\circ\mathfrak{t}_1\circ\Phi^{-1}=\mathfrak{t}_2\) for some automorphism \(\Phi\in\Aut(\mathfrak{g})\). Let \(g\colon\mathfrak{t}_1^*\to\mathfrak{t}_2^*\) be the linear isomorphism \(g=\Ad_{\Phi^{-1}}^*\) given by \(g(\alpha)(\delta) = \alpha(\Phi^{-1}\circ\delta\circ\Phi)\). Then
\begin{align*}
\Phi(V_\alpha) \amp= \{\Phi(X): \delta(X)=\alpha(\delta)X\,\forall \delta\in\mathfrak{t}_1 \}\\
\amp= \{Y: \Phi\circ\delta\circ\Phi^{-1}(Y)=\alpha(\delta)Y\,\forall \delta\in\mathfrak{t}_1 \}\\
\amp= \{Y: \eta(Y)=g(\alpha)(\eta)Y\,\forall \eta\in\mathfrak{t}_2 \} = W_{g(\alpha)}\text{.}
\end{align*}
Hence the gradings \(\mathcal{V}\) and \(\mathcal{W}\) are equivalent, as claimed.

Regarding \ref{enum-subtori}, suppose that  \(\mathfrak t_1 \subset \mathfrak t_2\). We claim that \(\mathcal{V}=f_*\mathcal{W}\) through the restriction map \(f\colon \mathfrak{t}_2^*\to \mathfrak{t}_1^*\), \(f(\beta)=\restr{\beta}{\mathfrak{t}_1}\). Indeed, fix a basis \(X_1,\ldots,X_n\) of \(\mathfrak{g}\) that diagonalizes the split torus \(\mathfrak{t}_2\) (and hence also the subtorus \(\mathfrak{t}_1\)). Let \(\beta_1,\ldots,\beta_n\in\mathfrak{t}_2^*\) be the maps defined by \(\delta(X_i)=\beta_i(\delta)X_i\) and define \(\alpha_i = \beta_i|_{\mathfrak{t_1}}\). By construction \(X_i\in W_{\beta_i}\), \(X_i\in V_{\alpha_i}\), and \(f(\beta_i)=\alpha_i\), proving that \(\mathcal{V}=f_*\mathcal{W}\).
\end{proof}
Finally, we observe that any {torsion-free grading} has a {realization} induced by a split torus.
\begin{lemma}\label{lemma-torus-grading-correspondence}
Let \(\mathcal V\) be a {torsion-free grading}. Then there exists a {split torus} \(\mathfrak{t}\) whose {induced \(\mathfrak{t}^*\)-grading} is a {realization} of \(\mathcal{V}\).
\end{lemma}
\begin{proof}\label{g:proof:idp16}
Let \(A\) be the torsion-free abelian grading group of the grading \(\mathcal{V}\colon \mathfrak{g}=\bigoplus_{\alpha\in A}V_\alpha\) and let \(A^*\) be the space of homomorphisms \(A\to F\), where \(F\) is the base field of \(\mathfrak{g}\). By reducing to the subgroup generated by the {weights}, we may assume \(A \) is isomorphic to \(\mathbb{Z}^m\) for some \(m \ge 1\). For each \(\varphi\in A^*\) define the linear map
\begin{equation*}
\delta_\varphi\colon\mathfrak{g}\to\mathfrak{g},\quad \delta_\varphi(X) = \varphi(\alpha)X\quad\forall X\in V_\alpha\text{.}
\end{equation*}
We claim that \(\mathfrak t = \{\delta_\varphi : \varphi \in A^*\} \) is a split torus that induces a realization for \(\mathcal V\). Indeed, a direct computation shows that all the maps \(\delta_\varphi\) are derivations. They are diagonalizable since by construction they are multiples of the identity on each layer \(V_\alpha\). Hence \(\mathfrak t\) is a split torus.

Let then \(\mathcal W :\mathfrak{g}=\bigoplus_{\beta\in \mathfrak t^*}W_\beta\) be the {\(\mathfrak{t}^*\)-grading} induced by \(\mathfrak t\). Denote by \(\Omega\) the support of \(\mathcal V\), and define a map \(f \colon  \Omega \to \mathfrak t^*\) by \(f(\alpha)(\delta_\varphi) = \varphi(\alpha)\). Then \(f\) is well-defined: if \(\varphi, \phi \in A^*\) are such that \(\delta_\varphi = \delta_\phi\), then by the definition of \(\mathfrak t\) we have \(\varphi(\alpha) = \phi(\alpha)\) for all weights \(\alpha \in \Omega\).

First, we show that \(V_\alpha\subset W_{f(\alpha)}\) for every \(\alpha \in A\). By the construction of the torus \(\mathfrak t\), for each \(X\in V_\alpha\) we have that
\begin{equation*}
\delta_\varphi(X) = \varphi(\alpha)X = f(\alpha)(\delta_\varphi)X \quad \forall \delta_\varphi \in \mathfrak t.
\end{equation*}
By the definition of the grading \(\mathcal W\), we then have \(X\in W_{f(\alpha)}\) and so \(V_\alpha\subset W_{f(\alpha)}\).

Next, we show that the map \(f\) is injective, which would prove that \(V_\alpha = W_{f(\alpha)}\) for all \(\alpha\in \Omega\) and so \(\mathcal W\) would be a realization of \(\mathcal V \), as claimed. Note that since \(A \) is isomorphic to \(\mathbb{Z}^m\), for every non-zero \(\alpha \in A\) there exists a homomorphism \(\varphi \in A^*\) such that \(\varphi(\alpha)\neq 0\). Therefore, if \(\alpha, \alpha' \in \Omega\) are such that \(f(\alpha) = f(\alpha')\), then by the construction of the map \(f\) we have
\begin{equation*}
\varphi(\alpha-\alpha') = \varphi(\alpha)-\varphi(\alpha') = f(\alpha)(\delta_\varphi)-f(\alpha')(\delta_\varphi) = 0
\end{equation*}
for every homomorphism \(\varphi\colon A\to F\). So \(\alpha = \alpha'\) and \(f\) is injective, proving that \(\mathcal W\) is a realization of \(\mathcal V \).
\end{proof}

\typeout{************************************************}
\typeout{Subsection 2.5 Maximal gradings}
\typeout{************************************************}

\subsection{Maximal gradings}\label{section-maximal-gradings}
We now present the notion of {maximal grading} using maximal {split tori} and prove that a maximal grading has the {universal property} of push-forwards (see {Proposition~{\ref{prop-max-grading}}}). The formulation through the derivation algebra will be convenient in the construction of maximal grading in {Subsection~{\ref{section-maximal-construction}}}. The universal property will be exploited in {Subsection~{\ref{subsection-torsion-free}}} where we give a method to construct all {gradings over torsion-free abelian groups} of a Lie algebra from a given maximal grading.
\begin{definition}\label{def-maximal-grading}
Let \(\mathfrak g\) be a Lie algebra. A \terminology{maximal grading} of \(\mathfrak g\) is the {universal realization} of the {grading induced} by a maximal {split torus} of \(\der(\mathfrak g)\).
\end{definition}
\begin{remark}\label{rmk-max-grading-unique}
The {maximal grading} of a Lie algebra is unique up to {equivalence}, since maximal {split tori} are all conjugate (see for instance, \cite[Theorem~15.2.6.]{springer-book}). Indeed, by {Lemma~{\ref{lemma-conjugates-and-subtori}}}\ref{enum-conj} the conjugacy implies that any two maximal split tori induce equivalent gradings, so also their universal realizations are equivalent.
\end{remark}
\begin{proposition}\label{prop-max-grading}
Let \(\mathcal W\) be a {maximal grading} of \(\mathfrak{g}\) and \(\mathcal{V}\) a grading of \(\mathfrak{g}\). Then \(\mathcal{V}\) is a {push-forward} grading of \(\mathcal W\).
\end{proposition}
\begin{proof}\label{g:proof:idp17}
Let \(\mathcal{V}'\) be the realization of \(\mathcal V\) as a {\(\mathfrak{t}^*\)-grading} induced by a {split torus} \(\mathfrak{t}\) given by {Lemma~{\ref{lemma-torus-grading-correspondence}}}.  Let \(\mathfrak{t}'\supset \mathfrak t \) be a maximal split torus in \(\der(\mathfrak{g})\) with induced grading \(\mathcal W'\). By {Lemma~{\ref{lemma-conjugates-and-subtori}}}.\ref{enum-subtori}, the grading \(\mathcal{V'}\) is a push-forward of \(\mathcal W'\). 

Since the maximal grading is unique up to equivalence by {Remark~{\ref{rmk-max-grading-unique}}}, we may assume that \(\mathcal W\) is the universal realization of \(\mathcal W'\). Therefore, since \(\mathcal V\) is a realization of \(\mathcal V'\), by {Lemma~{\ref{lemma-universal-of-coarsening}}} the grading \(\mathcal{V}\) is a push-forward of \(\mathcal{W}\).
\end{proof}
\begin{remark}\label{g:remark:idp18}
It follows from {Proposition~{\ref{prop-max-grading}}} and the discussion in \cite[Section~3.6]{Kochetov-2009-gradings_on_simple_lie_algebras_survey} that {maximal gradings} are {universal realizations} of fine gradings. In \cite[Definition~3.18]{Cornulier-2016-gradings}, maximal gradings are defined as the {gradings induced} by maximal split tori in the automorphism group \(\Aut(\mathfrak{g})\). \cite[Proposition~3.20]{Cornulier-2016-gradings} states that maximal gradings in the sense of \cite{Cornulier-2016-gradings} have a universal property equivalent to {Proposition~{\ref{prop-max-grading}}}, so by {Lemma~{\ref{lemma-pf-with-generating-weights}}} any such grading is maximal also in the sense of {Definition~{\ref{def-maximal-grading}}}. The maximal gradings considered in \cite{Favre-1973-system_de_poids} are the gradings induced by maximal split tori.
\end{remark}

\typeout{************************************************}
\typeout{Section 3 Constructions}
\typeout{************************************************}

\section{Constructions}\label{section-constructions}

\typeout{************************************************}
\typeout{Subsection 3.1 Enumeration of torsion-free gradings}
\typeout{************************************************}

\subsection{Enumeration of torsion-free gradings}\label{subsection-torsion-free}
Following the method suggested in \cite[Section~3.7]{Kochetov-2009-gradings_on_simple_lie_algebras_survey}, we now give a simple way to enumerate a complete (and finite) set of {universal realizations} of {gradings} of a Lie algebra using the {maximal grading}.

For the rest of this section, let \(\mathfrak g\) be a Lie algebra and let \(\mathcal W : \mathfrak g = \bigoplus_{n\in \mathbb{Z}^k}W_n \) be a {maximal grading}  of \(\mathfrak g\) with {weights} \(\Omega\). Denote by \(\Omega-\Omega\) the difference set \(\Omega- \Omega = \{n-m \,\mid\, n,m\in \Omega\}\). For a subset \(I \subset \Omega-\Omega \), let
\begin{equation*}
\pi_I \colon \mathbb{Z}^k \to \mathbb{Z}^k/\langle I\rangle
\end{equation*}
be the canonical projection. We define the finite set
\begin{equation}
\Gamma = \{{(\pi_I)_*\mathcal W}\mid I \subset \Omega - \Omega, \;\mathbb{Z}^k/\langle I\rangle\text{ is torsion-free} \}.\label{eq-Gamma}
\end{equation}

\begin{proposition}\label{prop-Quot-gradings-are-everything}
The set \(\Gamma\) is, up to {equivalence}, a complete set of {universal realizations} of {torsion-free gradings}  of \(\mathfrak g\).
\end{proposition}
\begin{proof}\label{g:proof:idp19}
Let \(\mathcal{V}\) be the universal realization of some torsion-free grading. Due to {Lemma~{\ref{lemma-universal-realization-Zk}}}, the {grading group} of \(\mathcal V\) is some \(\mathbb{Z}^m\). By {Proposition~{\ref{prop-max-grading}}}, there exists a homomorphism \(f\colon \mathbb{Z}^k\to \mathbb{Z}^m \) and an automorphism \(\Phi \in \Aut(\mathfrak g)\) such that \(\mathcal V= f_*\Phi(\mathcal W) \). Let
\begin{equation*}
I = \ker(f)\cap (\Omega-\Omega).
\end{equation*}
We are going to show that \(\mathcal V' = (\pi_I)_*(\mathcal W) \) is equivalent to \(\mathcal{V}\). Then, a posteriori, \(\mathbb{Z}^k/\langle I\rangle\) is torsion-free and we have \(\mathcal V' \in \Gamma\), proving the claim.

First, since \(\ker(\pi_I) = \langle I\rangle \subseteq \ker(f)\), by the universal property of quotients there exists a unique homomorphism \(\phi \colon  \mathbb{Z}^k/\langle I\rangle  \to \mathbb{Z}^m \) such that \(f = \phi \circ \pi_I\). In particular,
\begin{equation*}
\mathcal V = f_*\Phi(\mathcal W) = \phi_*(\pi_I)_*\Phi(\mathcal W) = \phi_*\Phi(\mathcal V')\text{,}
\end{equation*}
so \(\mathcal V\) is a push-forward grading of \(\mathcal V'\).

Secondly, since also \(\ker(f)\cap (\Omega-\Omega) = I \subseteq \ker(\pi_I)\cap (\Omega-\Omega) \), we deduce that \(\mathcal V \) and \(\Phi(\mathcal V')\) are realizations of the same grading. Since \(\mathcal V\) is a universal realization, it follows that \(\Phi(\mathcal V')\) is a {push-forward}  grading of \(\mathcal V\). Consequently, \(\mathcal V'\) is a push-forward grading of \(\mathcal V\). Since the grading group of a universal realization is generated by the weights, we get that the gradings \(\mathcal V\) and \(\mathcal V'\) are equivalent  by {Lemma~{\ref{lemma-pf-with-generating-weights}}}, as wanted.
\end{proof}
Notice that some of the \(\mathbb{Z}^k/\langle I\rangle\)-gradings in \(\Gamma\) are typically equivalent to each other. From the classification point of view, a more challenging task is to determine the equivalence classes once the set \(\Gamma\) is obtained. In low dimensions, naive methods are enough to separate non-equivalent gradings, and for equivalent ones the connecting automorphism can be found rather easily.

In \cite{software-lie_algebra_gradings} we give a representative from each equivalence class in \(\Gamma\) for every 6 dimensional nilpotent Lie algebra over \(\mathbb{C}\) and for an extensive class of 7 dimensional Lie algebras over \(\mathbb{C}\). The results and the methods for distinguishing the equivalence classes of the obtained gradings are described in more detail in {Subsection~{\ref{applications-classification}}}.

\typeout{************************************************}
\typeout{Subsection 3.2 Stratifications}
\typeout{************************************************}

\subsection{Stratifications}\label{section-stratifications}
\begin{definition}\label{def-stratification}
A \terminology{stratification} (a.k.a. \terminology{Carnot grading}) is a {\(\mathbb{Z}\)-grading} \(\mathfrak{g}=\bigoplus_{n\in\mathbb{Z}}V_n\) such that \(V_1\) generates \(\mathfrak{g}\) as a Lie algebra. A Lie algebra \(\mathfrak{g}\) is \terminology{stratifiable} if it admits a stratification.
\end{definition}
In this section we present the linear problem of constructing a stratification for a Lie algebra (or determining that one does not exist). The method is based on \cite[Lemma~3.10]{Cornulier-2016-gradings}, which gives the following characterization of stratifiable Lie algebras:
\begin{lemma}\label{lemma-stratification-by-derivation}
A nilpotent Lie algebra \(\mathfrak{g}\) is {stratifiable} if and only if there exists a derivation \(\delta\colon\mathfrak{g}\to\mathfrak{g}\) such that the induced map \(\mathfrak{g}/[\mathfrak{g},\mathfrak{g}]\to \mathfrak{g}/[\mathfrak{g},\mathfrak{g}]\) is the identity map. Moreover, a stratification is given by the {layers} \(V_i=\ker (\delta-i)\).
\end{lemma}
The condition of {Lemma~{\ref{lemma-stratification-by-derivation}}} is straightforward to check in a basis {adapted to the lower central series}.
\begin{definition}\label{def-basis-adapted-to-lcs}
Recall that the lower central series of a Lie algebra \(\mathfrak{g}\) is the decreasing sequence of ideals
\begin{equation*}
\mathfrak{g}=\mathfrak{g}^{(1)}\supset \mathfrak{g}^{(2)}\supset \mathfrak{g}^{(3)}\supset\cdots\text{,}
\end{equation*}
where \(\mathfrak{g}^{(i+1)} = [\mathfrak{g},\mathfrak{g}^{(i)}]\). A basis \(X_1,\dots,X_n\) of a Lie algebra \(\mathfrak{g}\) is \terminology{adapted to the lower central series} if for every non-zero \(\mathfrak{g}^{(i)}\) there exists an index \(n_i\in\mathbb{N}\) such that \(X_{n_i},\dots,X_n\) is a basis of \(\mathfrak{g}^{(i)}\). The \terminology{degree} of the basis element \(X_i\) is the integer \(w_i=\max\{j\in\mathbb{N}: X_i\in\mathfrak{g}^{(j)}\}\).
\end{definition}
\begin{proposition}\label{prop-stratification-explicit-leibniz}
Let \(X_1,\ldots,X_n\) be a basis {adapted to the lower central series} of a nilpotent Lie algebra \(\mathfrak{g}\) defined over a field \(F\). Let \(w_1,\ldots,w_n\) be the {degrees}  of the basis elements and let \(c_{ij}^k\in F\) be the structure coefficients in the basis. A linear map \(\delta\colon\mathfrak{g}\to\mathfrak{g}\) is a derivation that restricts to the identity on \(\mathfrak{g}/[\mathfrak{g},\mathfrak{g}]\) if and only if
\begin{equation}
\delta(X_i) = w_iX_i + \sum_{w_j\gt w_i}a_{ij}X_j\label{eq-derivation-in-lcs-basis}
\end{equation}
such that, for each triple of indices \(i,j,k\) such that \(w_k\gt w_i+w_j\), the coefficients \(a_{ij}\in F\) satisfy the linear equation
\begin{align}
c_{ij}^k(w_k-w_i-w_j) = \amp\sum_{w_i\lt w_h\leq w_k-w_j}a_{ih}c_{hj}^k+\sum_{w_j\lt w_h\leq w_k-w_i}a_{jh}c_{ih}^k\label{eq-stratifying-leibniz-condition}\\
-\amp\sum_{w_i+w_j\leq w_h\lt w_k}c_{ij}^ha_{hk}\text{.}\notag
\end{align}

\end{proposition}
\begin{proof}\label{g:proof:idp20}
If \(\delta\colon\mathfrak{g}\to\mathfrak{g}\) is a derivation that restricts to the identity on \(\mathfrak{g}/[\mathfrak{g},\mathfrak{g}]\), then by {Lemma~{\ref{lemma-stratification-by-derivation}}} \(\mathfrak{g}\) admits a {stratification}
\begin{equation*}
\mathfrak{g}=V_1\oplus\dots\oplus V_s
\end{equation*}
such that \(\restr{\delta}{V_i} = i\cdot \operatorname{id}\). Since the terms of the lower central series are given in terms of the stratification as \(\mathfrak{g}^{(i)}=V_i\oplus\dots\oplus V_s \), it follows that \(\delta(Y)\in i\cdot Y+\mathfrak{g}^{(i+1)}\) for any \(Y\in\mathfrak{g}^{(i)}\). That is, a derivation \(\delta\) restricting to the identity on \(\mathfrak{g}/[\mathfrak{g},\mathfrak{g}]\) is of the form {({\ref{eq-derivation-in-lcs-basis}})} for some coefficients \(a_{ij}\in F\).

It is then enough to show that {({\ref{eq-stratifying-leibniz-condition}})} is equivalent to the Leibniz rule
\begin{equation*}
\delta([X_i,X_j]) = [\delta(X_i),X_j]+[X_i,\delta(X_j)],\quad \forall i,j\in\{1,\ldots,n\}\text{.}
\end{equation*}
Indeed, this would prove that a linear map defined by {({\ref{eq-derivation-in-lcs-basis}})} is a derivation if and only if the coefficients \(a_{ij}\) satisfy the linear system {({\ref{eq-stratifying-leibniz-condition}})}.

Since the basis \(X_i\) is adapted to the lower central series, only the structure coefficients with large enough degrees are non-zero, i.e.\@, we have
\begin{equation}
[X_i,X_j] = \sum_{w_k\geq w_i+w_j}c_{ij}^kX_k\text{.}\label{eq-lcs-basis-structure-coefficients}
\end{equation}
By direct computation using {({\ref{eq-derivation-in-lcs-basis}})} and {({\ref{eq-lcs-basis-structure-coefficients}})} we get the expressions
\begin{align*}
[\delta(X_i),X_j] \amp =  \sum_{w_k\geq w_i+w_j}c_{ij}^kw_iX_k + \sum_{w_h\gt w_i}\sum_{w_k\geq w_h+w_j}a_{ih}c_{hj}^kX_k\\
[X_i,\delta(X_j)] \amp = \sum_{w_k\geq w_i+w_j}c_{ij}^kw_jX_k + \sum_{w_h\gt w_j}\sum_{w_k\geq w_i+w_h}a_{jh}c_{ih}^kX_k\\
\delta([X_i,X_j]) \amp = \sum_{w_k\geq w_i+w_j}c_{ij}^kw_kX_k + \sum_{w_h\geq w_i+w_j}\sum_{w_k\gt w_h}c_{ij}^ha_{hk}X_k
\end{align*}
Denoting \(\sum_k B_{ij}^kX_k = \delta([X_i,X_j])-[\delta(X_i),X_j]-[X_i,\delta(X_j)]\), we find that the equation \(B_{ij}^k=0\) is up to reorganizing terms equivalent to {({\ref{eq-stratifying-leibniz-condition}})}.

Finally, we observe that when \(w_k\leq w_i+w_j\), the condition \(B_{ij}^k=0\) is automatically satisfied: for \(w_k\lt w_i+w_j\) all of the sums are empty, and for \(w_k=w_i+w_j\), the only remaining terms from the sums cancel out as
\begin{equation*}
B_{ij}^k = c_{ij}^kw_k-c_{ij}^kw_i-c_{ij}^kw_j=0\text{.}\qedhere
\end{equation*}

\end{proof}
The concrete criterion of {Proposition~{\ref{prop-stratification-explicit-leibniz}}} provides the algorithm to construct a stratification. That is, a stratification or the non-existence of one for a nilpotent Lie algebra \(\mathfrak{g}\) is found as follows:
\begin{enumerate}
\item\label{alg-stratification-1-lcs}Construct a basis \(X_1,\ldots,X_n\) {adapted to the lower central series}.
\item\label{alg-stratification-2-leibniz}Find a derivation \(\delta\) as in {({\ref{eq-derivation-in-lcs-basis}})} by solving the linear system {({\ref{eq-stratifying-leibniz-condition}})}. If the system has no solutions, then \(\mathfrak{g}\) is not {stratifiable}.
\item\label{alg-stratification-3-layers}Return the stratification with the {layers} \(V_i=\ker (\delta-i)\).
\end{enumerate}

\begin{remark}\label{remark-stratification-field-extension}
By \cite[Theorem~1.4]{Cornulier-2016-gradings}, the existence of a stratification for a Lie algebra is invariant under base field extensions, so it suffices to work within any field \(F\) that \(\mathfrak{g}\) is {defined over}.
\end{remark}

\typeout{************************************************}
\typeout{Subsection 3.3 Positive gradings}
\typeout{************************************************}

\subsection{Positive gradings}\label{sssec-pos-gradings}
\begin{definition}\label{def-positive-grading}
An {\(\mathbb{R}\)-grading} \(\mathcal{V} \colon \mathfrak{g} = \bigoplus_{\alpha \in \mathbb{R}} V_\alpha \) is \terminology{positive} if \(\alpha\gt 0\) for all the {weights} of \(\mathcal{V}\). If such a grading exists for \(\mathfrak{g}\), then \(\mathfrak{g}\) is said to be \terminology{positively gradable}.
\end{definition}
In this section, we formulate and prove {Algorithm~{\ref{algorithm-positive-realization}}}. Using this algorithm one can decide whether a given grading of a Lie algebra admits a positive realization. If one starts with a Lie algebra with a known maximal grading, one is therefore able to answer the following questions:
\begin{enumerate}[label=(\roman*)]
\item\label{question-existence-of-positive-grading}Can the Lie algebra be equipped with a positive grading?
\item\label{question-enumeration-of-positive-gradings}Can one find in some sense all positive gradings of the Lie algebra?
\end{enumerate}

The methods of this article to construct a maximal grading are guaranteed to work only when the Lie algebra is defined over an algebraically closed field, see {Subsection~{\ref{section-maximal-construction}}}. In this discussion, we shall assume we are given a Lie algebra and a maximal grading for it, but we are not assuming that the field of coefficients is algebraically closed. However, regarding question~\ref{question-existence-of-positive-grading}, note that the existence of a positive grading for a given Lie algebra is invariant under extension of scalars by \cite[Theorem~1.4]{Cornulier-2016-gradings} so we may as well work with the algebraic closure.

To answer question~\ref{question-existence-of-positive-grading}, we observe that a Lie algebra admits a positive grading if and only if its maximal grading admits a positive realization by {Proposition~{\ref{prop-max-grading}}}. A maximal grading admits a positive realization exactly when the convex hull of its weights does not contain zero, see \cite[Proposition~3.22]{Cornulier-2016-gradings}. To concretely find a positive realization, one may use {Algorithm~{\ref{algorithm-positive-realization}}}.

Question~\ref{question-enumeration-of-positive-gradings} admits two relevant interpretations. First, one may use the enumeration of universal realizations of gradings of the given Lie algebra, as done in {Subsection~{\ref{subsection-torsion-free}}}, and using {Algorithm~{\ref{algorithm-positive-realization}}} construct their positive realizations when such realizations exist. The resulting list of positive gradings is complete in the sense that every positive grading of the given Lie algebra has the same layers as a grading on the list, up to a Lie algebra automorphism.

Question~\ref{question-enumeration-of-positive-gradings} may also be interpreted as finding a parametrization of the usually uncountable family of positive gradings. Let \(\mathcal W :\mathfrak{g}=\bigoplus_{n\in \mathbb{Z}^k}W_n\) be a {maximal grading} of our given Lie algebra \(\mathfrak g\) and let \(\Omega\) be the support of \(\mathcal{W}\). For any \(\mathbf{a}=(a_1,\ldots,a_k)\in\mathbb{R}^k\), let \(\pi^{\mathbf a} \colon \mathbb{Z}^k \to \mathbb{R}\) be the projection given by \(\pi^{\mathbf a}(e_i)=a_i\) with \(e_i\) denoting the standard basis elements of the lattice \(\mathbb{Z}^k\). Let
\begin{equation}
\positiveset = \{\mathbf{a}\in\mathbb{R}^k: \pi^{\mathbf a}(n)\gt 0\,\forall n\in\Omega \}\text{.}\label{eq-positive-set}
\end{equation}
The {push-forward grading} \(\pi^{\mathbf a}_*(\mathcal W)\) is a positive grading if and only if \(\mathbf{a}\in\positiveset\). Every positive grading of \(\mathfrak{g}\) is equivalent to some grading \(\pi^{\mathbf a}_*(\mathcal W)\) and hence corresponds to an element of the set \(\positiveset\). However, a pair of different elements of \(\positiveset\) may correspond to a pair of equivalent gradings.

We next present and prove {Algorithm~{\ref{algorithm-positive-realization}}}. The idea behind the algorithm is rather simple: finding a positive realization can be seen as a linear programming problem. The purpose of the slightly cumbersome form of the linear programming problem in {Algorithm~{\ref{algorithm-positive-realization}}} is to guarantee that the weights of the positive realization are small. This method works well for problems in small dimensions, but does not scale well to large problems. If one does not care about the resulting positive weights, there is a simpler algorithm, see {Remark~{\ref{rmk-lightweight-positivity-algorithm}}}.
\begin{algorithm}[Positive realization]\label{algorithm-positive-realization}
Input: A {torsion-free grading} \(\mathcal{V}\) for a Lie algebra \(\mathfrak{g}\). Output: A {positive} {}integer realization of \(\mathcal{V}\) with the smallest possible maximal weight, if any positive realization exists.
\begin{enumerate}
\item{}Compute the {universal realization} \(\widetilde{\mathcal{V}}\) of \(\mathcal{V}\) as in {Subsection~{\ref{ssec-univ-gradings}}}. Let \(\alpha_1,\ldots,\alpha_N \in \mathbb{Z}^k\) be the weights of \(\widetilde{\mathcal{V}}\).
\item\label{alg-pos-apriori-bound}Compute
\begin{equation*}
M = 1+\max\Big(\max_{i,j}\norm{\alpha_i-\alpha_j}_\infty,\max_i\norm{\alpha_i}_\infty\Big)
\end{equation*}
and set \(C = (3+N2^{N+1})M^k\).
\item\label{alg-pos-distinct-solution}Solve the integer linear programming problem
\begin{align}
\text{Minimize}\amp\amp z\label{eq-pos-linprog-optimization}\\
\text{subject to}\amp\amp z\geq \braket{w,\alpha_i} \amp\geq 1,\amp\quad 1\amp\leq i\leq N\label{eq-pos-linprog-positivity}\\
\amp\amp\braket{w,\alpha_i-\alpha_j} \amp\geq b_{ij} -C(1-b_{ij}),\amp\quad 1\amp\leq i\lt j\leq N\label{eq-pos-linprog-disjoint1}\\
\amp\amp\braket{w,\alpha_i-\alpha_j} \amp\leq b_{ij}-1 + Cb_{ij},\amp\quad 1\amp\leq i\lt j\leq N\label{eq-pos-linprog-disjoint2}
\end{align}
in the variables \(z\in\mathbb{Z}\), \(w\in\mathbb{Z}^k\) and binary variables \(b_{ij}\in\{0,1\}\). If no solution exists, then the grading \(\mathcal{V}\) does not have a positive realisation.
\item\label{alg-pos-pushforward}Let \(f\colon \mathbb{Z}^k\to\mathbb{Z}\) be the homomorphism \(f(\cdot)=\langle w,\cdot\rangle\). Return the {push-forward grading} \(f_*\widetilde{\mathcal{V}}\).
\end{enumerate}

\end{algorithm}
\begin{proof}[Proof of correctness]\label{g:proof:idp21}
If the grading \(\mathcal{V}\) has a positive realization, then it is a {push-forward grading} of the universal realization by some homomorphism \(f\colon \mathbb{Z}^k\to\mathbb{R}\) satisfying the inequalities \(f(\alpha_i)\gt 0\) and \(f(\alpha_i)\neq f(\alpha_j)\) for all \(i\neq j\). Since the inequalities all have integer coefficients, the existence of such a homomorphism is equivalent to the existence of a homomorphism \(f\colon \mathbb{Z}^k\to\mathbb{Z}\) with the same properties. We may always write such a homomorphism in the form \(f(\cdot)=\braket{w,\cdot}\) for some \(w\in\mathbb{Z}^k\). To prove the correctness of the algorithm, we need to show that the linear programming problem {({\ref{eq-pos-linprog-optimization}})}\textendash{}{({\ref{eq-pos-linprog-disjoint2}})} has a solution if and only if there exists \(w\in\mathbb{Z}^k\) such that
\begin{align}
\braket{w,\alpha_i}\amp\geq 1\label{eq-positive}\\
\braket{w,\alpha_i-\alpha_j}\amp\neq 0\label{eq-disjoint}
\end{align}
and that this solution has the smallest possible \(\max_i \braket{w,\alpha_i}\). Furthermore we claim, that if a suitable \(w\in\mathbb{Z}^k\) exists, then there also exists one with
\begin{equation}
\abs{\braket{w,\alpha_i-\alpha_j}}\leq C\label{eq-bounded-abs-diff}
\end{equation}
where \(C\) is the constant defined in step~\ref{alg-pos-apriori-bound}. We prove this claim later.

The smallest maximal weight property is equivalent to {({\ref{eq-pos-linprog-optimization}})} and the first half of {({\ref{eq-pos-linprog-positivity}})}, since a solution will necessarily satisfy \(z = \max_i \braket{w,\alpha_i}\). The latter half of {({\ref{eq-pos-linprog-positivity}})} is exactly the condition {({\ref{eq-positive}})}. The inequalities {({\ref{eq-disjoint}})} and {({\ref{eq-bounded-abs-diff}})} are encoded in the inequalities {({\ref{eq-pos-linprog-disjoint1}})} and {({\ref{eq-pos-linprog-disjoint2}})} using the auxiliary binary variables \(b_{ij}\). Indeed, if we have \(b_{ij}=0\), then the inequalities reduce to
\begin{equation*}
-C\leq \braket{w,\alpha_i-\alpha_j}\leq -1
\end{equation*}
and if \(b_{ij}=1\) then the inequalities reduce to
\begin{equation*}
1\leq \braket{w,\alpha_i-\alpha_j}\leq C\text{.}
\end{equation*}
Therefore it remains to prove the claim about  the additional condition {({\ref{eq-bounded-abs-diff}})}.

First we show that disregarding the other constraints, the system {({\ref{eq-positive}})} has a solution if and only if there exists a solution with \(\abs{\braket{w,\alpha_i}} \leq 1+N2^N\). The normal form of the system {({\ref{eq-positive}})} is given by switching to the variables \(x_i = \braket{w,\alpha_i}-1\), resulting in the system
\begin{equation}
Ax=d,\quad x\geq 0\text{,}\label{eq-pos-linprog-normal-form}
\end{equation}
where the matrix \(A\) is the matrix whose rows are \(e_i+e_j-e_k\in\mathbb{Z}^N\) for each linear relation \(\alpha_i+\alpha_j=\alpha_k\) (dropping linearly dependent conditions) and the right-hand-side vector is \(d=(-1,\ldots,-1)\).

The non-zero components of the basic feasible solutions of the normal form system {({\ref{eq-pos-linprog-normal-form}})} are determined by \(B^{-1}d\) where \(B\) is some invertible square submatrix of \(A\). Writing
\begin{equation*}
B^{-1} = \frac{1}{\det B}\operatorname{Adj}(B)\text{,}
\end{equation*}
where \(\operatorname{Adj}(B)\) is the adjugate matrix of \(B\), we see that integer solutions are determined by the vectors \(\operatorname{Adj}(B)d\). Since each row of \(A\) has the norm bound \(\norm{e_i+e_j-e_k}_\infty\leq 2\) every minor of \(A\) is bounded by \(2^N\). Hence we can bound the norms of integer basic feasible solutions to {({\ref{eq-pos-linprog-normal-form}})} by
\begin{equation*}
\norm{x}_\infty = \norm{\operatorname{Adj}(B)d}_\infty \leq N2^N\text{.}
\end{equation*}
Consequently the original problem {({\ref{eq-positive}})} has a solution \(w\in\mathbb{Z}^k\) if and only if there exists a solution \(w\) with
\begin{equation*}
\abs{\braket{w,\alpha_i}} = \abs{x_i+1} \leq 1+N2^N\text{.}
\end{equation*}

Finally, if \(w\in\mathbb{Z}^k\) is as above, we claim that \(\tilde{w} = M^{k}w + (1,M,\ldots,M^{k-1})\) is a solution to {({\ref{eq-positive}})}\textendash{}{({\ref{eq-bounded-abs-diff}})}.

To see that \(\tilde{w}\) satisfies {({\ref{eq-positive}})}\textendash{}{({\ref{eq-bounded-abs-diff}})}, we consider base-\(M\) expansions of the integers \(\braket{\tilde{w},\alpha_i}\) and \(\braket{\tilde{w},\alpha_i-\alpha_j}\). Since \(\norm{\alpha_i}_\infty\lt M\), we have
\begin{equation*}
\braket{\tilde{w},\alpha_i} = M^{k}\braket{w,\alpha_i} + \sum_{j=1}^{k}M^{j-1}\braket{e_j,\alpha_i}\geq M^{k} - \sum_{j=1}^{k}M^{j-1}\norm{\alpha_i}_\infty\geq 1.
\end{equation*}
A similar computation using \(\norm{\alpha_i-\alpha_j}_\infty\lt M\) gives the bound
\begin{align*}
\abs{\braket{\tilde{w},\alpha_i-\alpha_j}} \amp\leq M^{k}\abs{\braket{w,\alpha_i}}+M^{k}\abs{\braket{w,\alpha_j}} + \sum_{j=1}^{k}M^{j-1}\norm{\alpha_i-\alpha_j}_\infty\\
\amp\leq (2+N2^{N+1})M^{k} + M^{k}=C\text{,}
\end{align*}
showing {({\ref{eq-bounded-abs-diff}})} so it remains to verify {({\ref{eq-disjoint}})}. Expanding in terms of powers of \(M\), we have
\begin{equation*}
\braket{\tilde{w},\alpha_i-\alpha_j}  = \sum_{h=1}^{k}M^{h-1}\braket{e_h,\alpha_i-\alpha_j} \quad\mod{M^k}\text{.}
\end{equation*}
Since \(\abs{\braket{e_h,\alpha_i-\alpha_j}}\leq \norm{\alpha_i-\alpha_j}_\infty\lt M\) it follows that \(\braket{\tilde{w},\alpha_i-\alpha_j}\neq 0\) as soon as at least one \(\braket{e_h,\alpha_i-\alpha_j}\neq 0\). Since \(\alpha_i\neq \alpha_j\), this latter condition is always satisfied for some \(h\).
\end{proof}
\begin{remark}\label{rmk-lightweight-positivity-algorithm}
To obtain any positive realization, there is a much simpler polynomial time algorithm: Solve the linear programming problem
\begin{equation*}
\braket{w,\alpha_i}\geq 1,\quad i=1,\ldots,N
\end{equation*}
in the rational variables \(w\in\mathbb{Q}^k\) and rescale and perturb the solution to
\begin{equation*}
\tilde{w} = M^kw + (1,M,\ldots,M^{k-1})
\end{equation*}
as in the proof of correctness to guarantee distinct weights. Then the push-forward grading by \(f(\tilde{w},\cdot)\colon \mathbb{Z}^k\to\mathbb{Q}\) is again a positive realization of the original grading, but the resulting weights may be quite large.
\end{remark}

\typeout{************************************************}
\typeout{Subsection 3.4 Maximal gradings}
\typeout{************************************************}

\subsection{Maximal gradings}\label{section-maximal-construction}
In this section we provide an algorithm to construct a maximal grading for a Lie algebra \(\mathfrak{g}\) defined over an algebraically closed field \(F\).
\begin{algorithm}[Maximal grading]\label{algorithm-maximal-grading}
Input: A Lie algebra \(\mathfrak{g}\) defined over an algebraically closed field \(F\). Output: A {maximal grading} of \(\mathfrak{g}\).
\begin{enumerate}
\item\label{maxalg-1-derivations}Compute  the derivation algebra \(\der(\mathfrak{g})\). Set \(B=\emptyset\).
\item\label{maxalg-2-centralizer}Compute a basis \(A_1,\ldots,A_n\) for the centralizer \(C(B)\subset \der(\mathfrak{g})\).
\item\label{maxalg-3-extension}Repeat for each basis element \(A_i\), \(i=1,\ldots,n\): compute the Jordan decomposition \(A_i=S_i+N_i\). If the semisimple part \(S_i\) is not in the linear span of \(B\), extend \(B\) by \(S_i\) and go back to step~\ref{maxalg-2-centralizer}.
\item\label{maxalg-4-eigenspace-grading}Determine the \(\mathfrak{t}^*\)-grading \(\mathcal{V}:\mathfrak{g}=\bigoplus_\lambda V_\lambda\) {induced by the torus} \(\mathfrak{t}=\langle B\rangle\).
\item\label{maxalg-5-zk-indexing}Compute and return the {universal realization} of the grading \(\mathcal{V}\).
\end{enumerate}

\end{algorithm}
\begin{remark}\label{remark-maximal-grading-field-extension}
If the Lie algebra \(\mathfrak{g}\) is {defined over} a field \(F\), then \(\der(\mathfrak{g})\) has a maximal torus defined over \(F\). Hence the base field of \(\mathfrak{g}\) does not play a role in {Algorithm~{\ref{algorithm-maximal-grading}}}.
\end{remark}
The rest of the section is devoted to proving the correctness of {Algorithm~{\ref{algorithm-maximal-grading}}} and to explaining the steps in more detail.

Steps~\ref{maxalg-1-derivations} and \ref{maxalg-2-centralizer} are straightforward linear algebra. Step~\ref{maxalg-3-extension} is the core of the algorithm, where the basis \(B\) is extended until the spanned torus is maximal. Directly by construction each additional element \(S_i\in\der(\mathfrak{g})\) is a semisimple derivation that commutes with all the previous elements of \(B\), so \(B\) always spans a torus. The nontrivial part is that this construction guarantees that the resulting torus \(\mathfrak{t}\) spanned by \(B\) is maximal. This is guaranteed by the following lemma.
\begin{lemma}\label{lemma-torus-extension}
Let \(\mathfrak{g}\) be a Lie algebra defined over an algebraically closed field \(F\). Let \(\mathfrak{t}\subset\der(\mathfrak{g})\) be a torus and let \(A_1,\ldots,A_n\) be a basis of the centralizer \(C(\mathfrak{t})\subset \der(\mathfrak{g})\). Let \(A_i = S_i+N_i\) be the Jordan decompositions of each basis element. If \(S_i\in\mathfrak{t}\) for all \(i=1,\ldots,n\), then there do not exist any semisimple derivations in \(C(\mathfrak{t})\setminus\mathfrak{t}\).
\end{lemma}
\begin{proof}\label{g:proof:idp22}
First we claim that if \(S_i\in\mathfrak{t}\) for all \(i=1,\ldots,n\), then the centralizer \(C(\mathfrak{t})\) is a nilpotent Lie algebra. By Engel's theorem the centralizer is nilpotent if and only if each map \(\ad(A_i)\colon C(\mathfrak{t})\to C(\mathfrak{t})\) is nilpotent. By definition \(\mathfrak{t}\) is central in \(C(\mathfrak{t})\), so we have
\begin{equation*}
\ad(A_i)=\ad(S_i)+\ad(N_i) = \ad(N_i)\text{.}
\end{equation*}
Since each \(N_i\in\der(\mathfrak{g})\) is nilpotent, so is \(\ad(N_i)\) and the claim follows.

Next, we claim that the Jordan decomposition of a sum of basis elements is
\begin{equation}
A_i+A_j = (S_i+S_j) + (N_i+N_j)\text{.}\label{eq-jordan-decomp-of-sum}
\end{equation}
By assumption \(S_i,S_j\in\mathfrak{t}\), so also \(S_i+S_j\in\mathfrak{t}\) and hence the sum \(S_i+S_j\) is semisimple. Moreover since \(\mathfrak{t}\) is central, \([S_i+S_j,N_i+N_j]=0\), so all that remains is to show that \(N_i+N_j\) is nilpotent.

Since the centralizer \(C(\mathfrak{t})\) is nilpotent, it is also solvable. Since the field \(F\) is an algebraically closed field of characteristic zero, Lie's theorem implies that there exists a basis of \(\mathfrak{g}\) such that all the derivations \(A_i\) are represented by upper triangular matrices. Then \(N_i\) and \(N_j\) are both strictly upper triangular matrices, so also the sum \(N_i+N_j\) is strictly upper triangular, and hence nilpotent.

The Jordan decompositions {({\ref{eq-jordan-decomp-of-sum}})} and the assumption that \(S_i\in\mathfrak{t}\) for all \(i=1,\ldots,n\) imply that the semisimple part of every linear combination of the elements \(A_i\) is also contained in \(\mathfrak{t}\). Hence there cannot exist any semisimple elements in \(C(\mathfrak{t})\setminus\mathfrak{t}\).
\end{proof}
\begin{remark}\label{remark-jordan-decompositions}
The Jordan decompositions required in Step~\ref{maxalg-3-extension} of {Algorithm~{\ref{algorithm-maximal-grading}}} can be efficiently computed using the algorithm given in Appendix~A.2 of \cite{deGraaf-2000-Lie_algebras_theory_and_algorithms}.
\end{remark}
In step~\ref{maxalg-4-eigenspace-grading}, the grading induced by the torus \(\mathfrak{t}\) has a concrete description in terms of the fixed basis \(B\) of \(\mathfrak{t}\). Namely, the basis \(\delta_1,\ldots,\delta_k\) defines an isomorphism \(\mathfrak{t}^*\to F^k\) and hence an {equivalent} {push-forward grading} over \(F^k\). Expanding out the construction of {Lemma~{\ref{lemma-gradings-induced}}} shows that the push-forward grading has the {layers}
\begin{equation*}
V_\lambda = V_{(\lambda_1,\ldots,\lambda_k)}
= \bigcap_{i=1}^k E^{\lambda_i}_{\delta_i}\text{,}
\end{equation*}
where \(E^{\lambda_i}_{\delta_i}\) is the (possibly zero) eigenspace for the eigenvalue \(\lambda_i\) of the derivation \(\delta_i\).

The final part of {Algorithm~{\ref{algorithm-maximal-grading}}} is step~\ref{maxalg-5-zk-indexing}, where we replace the indexing by eigenvalues of the derivations of \(\mathfrak{t}\) with indexing over some \(\mathbb{Z}^k\) given by the universal realization. The precise method was described earlier in {Subsection~{\ref{ssec-univ-gradings}}}. Since the construction of the first three steps of {Algorithm~{\ref{algorithm-maximal-grading}}} leads to a maximal torus of \(\der(\mathfrak{g})\), by {Definition~{\ref{def-maximal-grading}}} the output is a maximal grading of \(\mathfrak{g}\).
\begin{remark}\label{remark-non-algebraically-closed-fields}
The relevance of the assumption that the field \(F\) is algebraically closed is to guarantee that the constructed tori are split, i.e.\@, the semisimple derivations are diagonalizable. The Jordan decomposition and {Lemma~{\ref{lemma-torus-extension}}} then give us an efficient method to construct diagonalizable derivations in \(C(\mathfrak{t})\setminus\mathfrak{t}\).

When the Lie algebra \(\mathfrak{g}\) is defined over a non-algebraically closed field \(F\), it is possible that the computation of a maximal grading over the algebraic closure \(\bar{F}\) using {Algorithm~{\ref{algorithm-maximal-grading}}} outputs a grading that is still defined over \(F\). Then the output is also a maximal grading of the Lie algebra \(\mathfrak{g}\) over \(F\). Consider for example the Lie algebras denoted as \(L_{6,19}(\epsilon)\) in \cite{Cicalo-deGraaf-Schneider-2012-6d_nilpotent_lie_algebras}, which are 6-dimensional Lie algebras defined by the structure coefficients
\begin{align*}
[X_1,X_2] \amp= X_4\amp
[X_1,X_3] \amp= X_5\\
[X_1,X_5] \amp=
[X_2,X_4] = X_6\amp
[X_3,X_5] \amp= \epsilon X_6
\end{align*}
For the Lie algebra \(L_{6,19}(-1)\) the maximal torus computed by {Algorithm~{\ref{algorithm-maximal-grading}}} over the algebraic numbers is also defined over \(\mathbb{Q}\), but this is not the case with the Lie algebra \(L_{6,19}(1)\). Indeed for \(L_{6,19}(1)\), the maximal torus over the rationals is 2-dimensional, but the maximal torus over the algebraic numbers is 3-dimensional.
\end{remark}

\typeout{************************************************}
\typeout{Section 4 Applications}
\typeout{************************************************}

\section{Applications}\label{section-applications}

\typeout{************************************************}
\typeout{Subsection 4.1 Structure from maximal gradings}
\typeout{************************************************}

\subsection{Structure from maximal gradings}\label{applications-structure}
In this subsection we show how {maximal gradings} may be used to find some structural information of Lie algebras. We start by studying how maximal gradings reveal the structure of a direct product. A similar result can be found in 1.6.5 of  \cite{Favre-1973-system_de_poids}.
\begin{example}\label{ex-Heis-times-Heis}
Consider the Lie algebra \(L_{6,22}(1)\) in \cite{Cicalo-deGraaf-Schneider-2012-6d_nilpotent_lie_algebras} with basis \(\{X_1,\dots,X_6 \}\), where the only non-zero bracket relations are
\begin{equation*}
[X_1,X_2]=X_5,\;\; [X_1,X_3]=X_6,\;\; [X_2,X_4]=X_6,\;\; [X_3,X_4]=X_5\text{.}
\end{equation*}
In a basis \(\{Y_1,\dots,Y_6 \}\) {adapted} to the {maximal grading}, the bracket relations are
\begin{equation*}
[Y_1,Y_2]=Y_3,\;\;[Y_4,Y_5]=Y_6\text{.}
\end{equation*}
From these bracket relations we immediately see that the Lie algebra \(L_{6,22}(1)\) is isomorphic to \(L_{3,2} \times L_{3,2}\), where \(L_{3,2}\) is the first Heisenberg Lie algebra.
\end{example}
\label{par-non-degenerate}
We say that a {split torus} \(\mathfrak{t}\subset\der(\mathfrak{g})\) is \terminology{non-degenerate} if the intersection of the kernels of the maps \(D \in \mathfrak{t}\) is trivial. That is, a split torus is non-degenerate if and only if the {\(\mathfrak{t}^*\)-grading} it {induces} does not have zero as a {weight}.

We expect that the following result is known even without the non-degeneracy assumption, however we have been unable to locate a reference. We will therefore give a direct proof of the simpler claim.
\begin{lemma}\label{max-torus-product-kernel}
Let \(\mathfrak{t}_1\subset\der(\mathfrak{g}_1)\) and \(\mathfrak{t}_2\subset\der(\mathfrak{g}_2)\) be non-degenerate maximal {split tori}. Then \(\mathfrak{t}_1 \times \mathfrak{t}_2\) is a maximal split torus in \(\der(\mathfrak{g}_1 \times \mathfrak{g}_2)\).
\end{lemma}
\begin{proof}\label{g:proof:idp23}
Denoting \(\mathfrak{t} = \mathfrak{t}_1 \times \mathfrak{t}_2\), let \(D \in C(\mathfrak{t})\) be a diagonalizable derivation in the centralizer \(C(\mathfrak{t})\). To show the maximality of \(\mathfrak{t}\), it suffices to show that \(D \in \mathfrak{t}\). In a basis adapted to the product we may represent
\begin{equation*}
D = \begin{bmatrix}
E_1 \amp F_1 \\ F_2 \amp E_2
\end{bmatrix}\text{,}
\end{equation*}
where \(E_1\in \der(\mathfrak{g}_1)\), \(E_2\in \der(\mathfrak{g}_2)\), and \(F_1\colon \mathfrak{g}_2\to\mathfrak{g}_1\) and \(F_2\colon \mathfrak{g}_1\to \mathfrak{g}_2\) are some linear maps. We are going to demonstrate that \(E_1\in \mathfrak t_1\), \(E_2\in \mathfrak t_2\) and \(F_1=F_2=0\), which would prove that \(D = E_1\times E_2 \in \mathfrak t\).

Let \(D_1 \in \mathfrak t_1\). By assumption \(D\) commutes with \(D_1\times 0 \in \mathfrak t\), so a simple computation shows that \(E_1\) commutes with \(D_1\) and \(D_1 F_1 = 0\). Since \(D_1\) is arbitrary, we obtain \(E_1 \in C(\mathfrak t_1) \). From the fact that \(D_1F_1 = 0\) for every \(D_1 \in \mathfrak t_1\) we get
\begin{equation*}
\operatorname{Im}(F_1) \subset  \bigcap_{D_1\in \mathfrak t_1} \ker(D_1) = \{0\},
\end{equation*}
where the last equality follows from the non-degeneracy of \(\mathfrak t_1\). Consequently, \(F_1 = 0\).

A similar argument shows that \(E_2 \in C(\mathfrak t_2)\) and \(F_2=0\). Since \(D\) is assumed diagonalizable, it follows that \(E_1\) and \(E_2\) are diagonalizable. Then by maximality of \(\mathfrak t_1\) and \(\mathfrak t_2\) we have \(E_1\in \mathfrak t_1\) and \(E_2\in \mathfrak t_2\), which shows that \(D = E_1\times E_2 \in \mathfrak t\).
\end{proof}
For {gradings}, the above lemma implies the following. Suppose \(\mathcal V : \mathfrak g_1 = \bigoplus_{\alpha \in A } V_\alpha\) and \(\mathcal W : \mathfrak g_2 = \bigoplus_{\beta \in B } W_\beta\) are {maximal gradings} of Lie algebras \(\mathfrak g_1\) and \(\mathfrak g_2\), and suppose zero is not a {weight} for either \(\mathcal{V}\) or \(\mathcal{W}\). Then
\begin{equation}
\mathcal V \times \mathcal W : \Bigl(\bigoplus_{(\alpha,0) \in A \times B} V_\alpha \times \{0\} \Bigr) \oplus
\Bigl( \bigoplus_{(0,\beta) \in A \times B} \{0\} \times W_\beta \Bigr)\label{eq-product-grading}
\end{equation}
is a maximal grading of \(\mathfrak g = \mathfrak g_1 \times \mathfrak g_2 \). Indeed, the gradings \(\mathcal{V}\) and \(\mathcal{W}\) are the {universal realizations} of gradings {induced} by the respective maximal {split tori} \(\mathfrak{t}_1\) and \(\mathfrak{t}_2\) of the Lie algebras \(\mathfrak{g}_1\) and \(\mathfrak{g}_2\). By {Lemma~{\ref{max-torus-product-kernel}}}, the product torus \(\mathfrak t_1 \times \mathfrak t_2\) is maximal. The universal realization of the grading induced by \(\mathfrak t_1 \times \mathfrak t_2\) is {equivalent} to the product grading {({\ref{eq-product-grading}})}.

\label{par-detect-product-struct}
Conversely, we can detect when a grading is a product grading. Similarly as in 1.6.4 of  \cite{Favre-1973-system_de_poids}, for a grading \(\mathcal{V}\colon \mathfrak{g} = \bigoplus_{\alpha \in A} V_\alpha \) with weights \(\Omega\subset A\) consider the graph with vertices \(\Omega\) defined as follows: Whenever \([V_\alpha, V_\beta]\neq 0\), we define edges between all the three vertices \(\alpha,\beta,\alpha+\beta\in \Omega\). If the graph \(\Omega\) admits a partition \(\Omega = \Omega_1 \sqcup \Omega_2\) such that no edges exist between \(\Omega_1\) and \(\Omega_2\), then the Lie algebra \(\mathfrak{g}\) is a direct product of the ideals \(\mathfrak{g}_1 = \bigoplus_{\alpha \in \Omega_1} V_\alpha\) and \(\mathfrak{g}_2 = \bigoplus_{\beta \in \Omega_2} V_\beta\). In this situation we say the grading \(\mathcal{V}\) \terminology{detects the product structure} \(\mathfrak{g}_1 \times \mathfrak{g}_2 \) of the Lie algebra \(\mathfrak{g}\). We gather the observations made above into the following proposition.
\begin{proposition}\label{prop-max-grading-detects-product}
If a Lie algebra \(\mathfrak{g}\) is decomposable and the  {maximal gradings} of the factor Lie algebras do not have zero as a {weight}, then the maximal grading of \(\mathfrak{g}\) detects the product structure.
\end{proposition}
We remark that while maximal gradings are able to detect product structures as indicated above, they are not able to detect some other algebraic properties. The Lie algebra \(L_{6,24}(1)\) in \cite{Cicalo-deGraaf-Schneider-2012-6d_nilpotent_lie_algebras} provides examples of two such phenomena. First, the layers of its maximal grading are not contained in the terms of its lower central series (this behavior can be also achieved by examples where the maximal grading is very coarse). Secondly, this Lie algebra has a \terminology{nice basis} (see \cite{conti-rossi-2019-nice_lie_groups} for the precise definition and its motivation), but it can be shown that no basis {adapted} to a maximal grading is nice.

Despite these negative results, maximal gradings have another structural application in simplifying the problem of deciding whether two Lie algebras are isomorphic or not.
\begin{remark}\label{rmk-max-graded-automorphisms}
If two Lie algebras \(\mathfrak g_1 \) and \(\mathfrak g_2\) are isomorphic, then any isomorphism maps the {maximal grading} of \(\mathfrak{g}_1\) to a maximal grading of \(\mathfrak{g}_2\). Therefore, if the maximal gradings of \(\mathfrak{g}_1\) and \(\mathfrak{g}_2\) are given, then deciding if \(\mathfrak{g}_1\) and \(\mathfrak{g}_2\) are isomorphic reduces to determining the existence of an isomorphism between the maximal gradings. In many cases this is significantly easier than naively solving the original isomorphism problem. For example, in low dimensions, the majority of the {layers} of the maximal grading are one-dimensional, in which case searching for possible isomorphisms becomes a combinatorial problem.
\end{remark}

\typeout{************************************************}
\typeout{Subsection 4.2 Classification of gradings in low dimension}
\typeout{************************************************}

\subsection{Classification of gradings in low dimension}\label{applications-classification}
Following the strategy outlined in \cite[Section~3.7]{Kochetov-2009-gradings_on_simple_lie_algebras_survey}, we classify {universal realizations} of {torsion-free gradings} in nilpotent Lie algebras of dimension up to 7 over \(\mathbb{C}\), apart from a few uncountable families of 7 dimensional Lie algebras. Regarding the uncountable families, we follow the study carried out in \cite{magnin-huge-book} and focus on those singular values of the complex parameter \(\lambda\) for which either the Lie algebra cohomology or the adjoint cohomology have different dimensions compared to the rest of the Lie algebras in the same family. We also include a few examples corresponding to non-singular values. The complete classification of the gradings can be found in \cite{software-lie_algebra_gradings}; here we will give a brief overview of Lie algebras of dimension up to 6.

The main part of the classification is the construction of a {maximal grading} ({Algorithm~{\ref{algorithm-maximal-grading}}}) and the enumeration of torsion-free gradings ({Proposition~{\ref{prop-Quot-gradings-are-everything}}}). Since all the Lie algebras we consider are {defined over} the algebraic numbers, we are free to make use of our computer algebra implementation, as discussed in {Subsection~{\ref{ssec-base-fields}}}. As a starting point we used the classifications of nilpotent Lie algebras given in \cite{deGraaf-2007-dim_6_nilpotent_lie_algebras} for dimensions less than 6, \cite{Cicalo-deGraaf-Schneider-2012-6d_nilpotent_lie_algebras} for dimension 6, and \cite{gong} for dimension 7. The classification up to dimension 6 has a pre-existing computer implementation in the GAP package \cite{LieAlgDB2.2}. Since these Lie algebras are not always given in a basis adapted to any maximal grading, we first compute the maximal grading using the methods described in {Subsection~{\ref{section-maximal-construction}}} and switch to a basis adapted to the resulting grading.

The presentations we use for the nilpotent Lie algebras up to dimension 6 are listed in {Table~{\ref{table-maximal-bases}}}. The Lie brackets \([Y_a,Y_b]=Y_c\) are listed in the condensed form \(ab=c\). Lie algebras \(\mathfrak{g}\times\mathbb{C}^d\) with abelian factors have identical structure coefficients with the nonabelian factor \(\mathfrak{g}\) and are omitted from the list. For example \(L_{4,2}=L_{3,2}\times\mathbb{C}\) has the basis \(Y_1,\ldots,Y_4\) with the bracket relation \([Y_1,Y_2]=Y_3\) from \(L_{3,2}\).
\begin{table}[hbtp]
\centering
{
\begin{tabular}[t]{llllllll}
\(L_{3,2}\)&\(12=3\)&&&&&&\tabularnewline[0pt]
\(L_{4,3}\)&\(12=3\)&\(13=4\)&&&&&\tabularnewline[0pt]
\(L_{5,4}\)&\(41=5\)&\(23=5\)&&&&&\tabularnewline[0pt]
\(L_{5,5}\)&\(13=4\)&\(14=5\)&\(32=5\)&&&&\tabularnewline[0pt]
\(L_{5,6}\)&\(12=3\)&\(13=4\)&\(14=5\)&\(23=5\)&&&\tabularnewline[0pt]
\(L_{5,7}\)&\(12=3\)&\(13=4\)&\(14=5\)&&&&\tabularnewline[0pt]
\(L_{5,8}\)&\(12=3\)&\(14=5\)&&&&&\tabularnewline[0pt]
\(L_{5,9}\)&\(12=3\)&\(23=4\)&\(13=5\)&&&&\tabularnewline[0pt]
\(L_{6,10}\)&\(23=4\)&\(51=6\)&\(24=6\)&&&&\tabularnewline[0pt]
\(L_{6,11}\)&\(12=3\)&\(13=5\)&\(15=6\)&\(23=6\)&\(24=6\)&&\tabularnewline[0pt]
\(L_{6,12}\)&\(23=4\)&\(24=5\)&\(31=6\)&\(25=6\)&&&\tabularnewline[0pt]
\(L_{6,13}\)&\(13=4\)&\(14=5\)&\(32=5\)&\(15=6\)&\(42=6\)&&\tabularnewline[0pt]
\(L_{6,14}\)&\(12=3\)&\(13=4\)&\(14=5\)&\(23=5\)&\(25=6\)&\(43=6\)&\tabularnewline[0pt]
\(L_{6,15}\)&\(12=3\)&\(13=4\)&\(14=5\)&\(23=5\)&\(15=6\)&\(24=6\)&\tabularnewline[0pt]
\(L_{6,16}\)&\(12=3\)&\(13=4\)&\(14=5\)&\(25=6\)&\(43=6\)&&\tabularnewline[0pt]
\(L_{6,17}\)&\(21=3\)&\(23=4\)&\(24=5\)&\(13=6\)&\(25=6\)&&\tabularnewline[0pt]
\(L_{6,18}\)&\(12=3\)&\(13=4\)&\(14=5\)&\(15=6\)&&&\tabularnewline[0pt]
\(L_{6,19}(-1)\)&\(12=3\)&\(14=5\)&\(25=6\)&\(43=6\)&&&\tabularnewline[0pt]
\(L_{6,20}\)&\(12=3\)&\(14=5\)&\(15=6\)&\(23=6\)&&&\tabularnewline[0pt]
\(L_{6,21}(-1)\)&\(12=3\)&\(23=4\)&\(13=5\)&\(14=6\)&\(25=6\)&&\tabularnewline[0pt]
\(L_{6,22}(0)\)&\(24=5\)&\(41=6\)&\(23=6\)&&&&\tabularnewline[0pt]
\(L_{6,22}(1)\)&\(12=3\)&\(45=6\)&&&&&\tabularnewline[0pt]
\(L_{6,23}\)&\(12=3\)&\(14=5\)&\(15=6\)&\(42=6\)&&&\tabularnewline[0pt]
\(L_{6,24}(0)\)&\(13=4\)&\(34=5\)&\(14=6\)&\(32=6\)&&&\tabularnewline[0pt]
\(L_{6,24}(1)\)&\(12=3\)&\(23=5\)&\(24=5\)&\(13=6\)&&&\tabularnewline[0pt]
\(L_{6,25}\)&\(12=3\)&\(13=4\)&\(15=6\)&&&&\tabularnewline[0pt]
\(L_{6,26}\)&\(12=3\)&\(24=5\)&\(14=6\)&&&&\tabularnewline[0pt]
\(L_{6,27}\)&\(12=3\)&\(13=4\)&\(25=6\)&&&&\tabularnewline[0pt]
\(L_{6,28}\)&\(12=3\)&\(23=4\)&\(13=5\)&\(15=6\)&&&
\end{tabular}
}

\caption{Lie algebras of dimension up to 6 over \(\mathbb{C}\) in a basis adapted to a maximal grading.}
\label{table-maximal-bases}
\end{table}

\label{par-rank-and-type}
With all the maximal gradings computed, we enumerate universal realizations of all torsion-free gradings as in {Proposition~{\ref{prop-Quot-gradings-are-everything}}}. For the classification up to equivalence, we first introduce some easy-to-check invariants for gradings. Recall that by {Lemma~{\ref{lemma-universal-realization-Zk}}}, the grading groups of the obtained gradings are isomorphic to some groups \(\mathbb{Z}^k\). The dimension \(k\) is called the \terminology{rank} of the grading. We recall also an invariant from \cite[Section~3.2]{Kochetov-2009-gradings_on_simple_lie_algebras_survey}: the \terminology{type} of a grading is the tuple \((n_1,n_2,\ldots,n_k)\), where \(k\) is the dimension of the largest layer, and each \(n_i\) is the number of \(i\)-dimensional layers.

From the full list of {torsion-free gradings}, we initially collect together {gradings} using the following criteria:
\begin{enumerate}
\item{}The ranks of the gradings are equal.
\item{}The types of the gradings are equal.
\item{}There exists a homomorphism between the grading groups of the {universal realizations} mapping layers to layers of equal dimensions.
\end{enumerate}
In this way we get for each Lie algebra families \(I_1,I_2,\ldots,I_k\) of gradings such that the gradings of \(I_i\) and \(I_j\) are not {equivalent} for \(i\neq j\).

To compute the precise equivalence classes, we naively check if the gradings within each family \(I_i\) are equivalent. For each pair of \(\mathbb{Z}^k\)-gradings \(\mathfrak{g}=\bigoplus_{\alpha\in \mathbb{Z}^k} V_\alpha\) and \(\mathfrak{g}=\bigoplus_{\beta\in \mathbb{Z}^k} W_\beta\), there are usually only a few homomorphisms \(f\colon \mathbb{Z}^k\to \mathbb{Z}^k\) with \(\dim V_\alpha=\dim W_{f(\beta)}\). For each such homomorphism \(f\), we need to check whether there exists an automorphism \(\Phi\in\Aut(\mathfrak{g})\) such that \(\Phi(V_\alpha)=W_{f(\beta)}\) for all weights \(\alpha\). These identities define a system of quadratic equations over algebraic numbers. Since we are working over an algebraically closed field, the system has no solution if and only if 1 is contained in the ideal defined by the polynomial equations. The dimensions of the layers are generally quite small in the cases we need to check, so Gröbner basis methods work well.

For nilpotent Lie algebras of dimension up to 6, an overview of our classification of gradings is compiled in {Table~{\ref{table-grading-info}}}. For each Lie algebra, we list its label in the classification of \cite{Cicalo-deGraaf-Schneider-2012-6d_nilpotent_lie_algebras}, the rank of its {maximal grading} (\(k\)), whether it is {stratifiable} or not (s?), the number of {gradings} (\#), and the number of gradings with a {positive} {realization} (\(\#\mathbb{Z}_+\)).
\begin{table}[hbtp]
\centering
{
\begin{tabular}[t]{lllll}
Name&\(k\)&s?&\#&\#\(\mathbb{Z}_+\)\tabularnewline\hline
\(L_{2,1}\)&\(2\)&\(\checkmark\)&\(2\)&\(2\)\tabularnewline[0pt]
\(L_{3,1}\)&\(3\)&\(\checkmark\)&\(3\)&\(3\)\tabularnewline[0pt]
\(L_{3,2}\)&\(2\)&\(\checkmark\)&\(4\)&\(2\)\tabularnewline[0pt]
\(L_{4,1}\)&\(4\)&\(\checkmark\)&\(5\)&\(5\)\tabularnewline[0pt]
\(L_{4,2}\)&\(3\)&\(\checkmark\)&\(11\)&\(6\)\tabularnewline[0pt]
\(L_{4,3}\)&\(2\)&\(\checkmark\)&\(6\)&\(2\)\tabularnewline[0pt]
\(L_{5,1}\)&\(5\)&\(\checkmark\)&\(7\)&\(7\)\tabularnewline[0pt]
\(L_{5,2}\)&\(4\)&\(\checkmark\)&\(26\)&\(15\)\tabularnewline[0pt]
\(L_{5,3}\)&\(3\)&\(\checkmark\)&\(22\)&\(9\)\tabularnewline[0pt]
\(L_{5,4}\)&\(3\)&\(\checkmark\)&\(9\)&\(4\)\tabularnewline[0pt]
\(L_{5,5}\)&\(2\)&\(\)&\(7\)&\(3\)\tabularnewline[0pt]
\(L_{5,6}\)&\(1\)&\(\)&\(2\)&\(1\)\tabularnewline[0pt]
\(L_{5,7}\)&\(2\)&\(\checkmark\)&\(7\)&\(2\)\tabularnewline[0pt]
\(L_{5,8}\)&\(3\)&\(\checkmark\)&\(14\)&\(6\)\tabularnewline[0pt]
\(L_{5,9}\)&\(2\)&\(\checkmark\)&\(5\)&\(2\)\tabularnewline[0pt]
\(L_{6,1}\)&\(6\)&\(\checkmark\)&\(11\)&\(11\)\tabularnewline[0pt]
\(L_{6,2}\)&\(5\)&\(\checkmark\)&\(52\)&\(31\)\tabularnewline[0pt]
\(L_{6,3}\)&\(4\)&\(\checkmark\)&\(60\)&\(27\)\tabularnewline[0pt]
\(L_{6,4}\)&\(4\)&\(\checkmark\)&\(29\)&\(13\)\tabularnewline[0pt]
\(L_{6,5}\)&\(3\)&\(\)&\(29\)&\(15\)\tabularnewline[0pt]
\(L_{6,6}\)&\(2\)&\(\)&\(8\)&\(6\)\tabularnewline[0pt]
\(L_{6,7}\)&\(3\)&\(\checkmark\)&\(31\)&\(11\)\tabularnewline[0pt]
\(L_{6,8}\)&\(4\)&\(\checkmark\)&\(52\)&\(25\)
\end{tabular}\hspace{1em}
\begin{tabular}[t]{lllll}
Name&\(k\)&s?&\#&\#\(\mathbb{Z}_+\)\tabularnewline\hline
\(L_{6,9}\)&\(3\)&\(\checkmark\)&\(17\)&\(8\)\tabularnewline[0pt]
\(L_{6,10}\)&\(3\)&\(\)&\(23\)&\(8\)\tabularnewline[0pt]
\(L_{6,11}\)&\(1\)&\(\)&\(2\)&\(1\)\tabularnewline[0pt]
\(L_{6,12}\)&\(2\)&\(\)&\(9\)&\(4\)\tabularnewline[0pt]
\(L_{6,13}\)&\(2\)&\(\)&\(8\)&\(3\)\tabularnewline[0pt]
\(L_{6,14}\)&\(1\)&\(\)&\(2\)&\(1\)\tabularnewline[0pt]
\(L_{6,15}\)&\(1\)&\(\)&\(2\)&\(1\)\tabularnewline[0pt]
\(L_{6,16}\)&\(2\)&\(\checkmark\)&\(8\)&\(2\)\tabularnewline[0pt]
\(L_{6,17}\)&\(1\)&\(\)&\(2\)&\(1\)\tabularnewline[0pt]
\(L_{6,18}\)&\(2\)&\(\checkmark\)&\(8\)&\(2\)\tabularnewline[0pt]
\(L_{6,19}(-1)\)&\(3\)&\(\checkmark\)&\(21\)&\(6\)\tabularnewline[0pt]
\(L_{6,20}\)&\(2\)&\(\checkmark\)&\(8\)&\(3\)\tabularnewline[0pt]
\(L_{6,21}(-1)\)&\(2\)&\(\checkmark\)&\(6\)&\(2\)\tabularnewline[0pt]
\(L_{6,22}(0)\)&\(3\)&\(\checkmark\)&\(18\)&\(8\)\tabularnewline[0pt]
\(L_{6,22}(1)\)&\(4\)&\(\checkmark\)&\(32\)&\(15\)\tabularnewline[0pt]
\(L_{6,23}\)&\(2\)&\(\)&\(8\)&\(4\)\tabularnewline[0pt]
\(L_{6,24}(0)\)&\(2\)&\(\)&\(8\)&\(4\)\tabularnewline[0pt]
\(L_{6,24}(1)\)&\(2\)&\(\)&\(5\)&\(2\)\tabularnewline[0pt]
\(L_{6,25}\)&\(3\)&\(\checkmark\)&\(29\)&\(11\)\tabularnewline[0pt]
\(L_{6,26}\)&\(3\)&\(\checkmark\)&\(10\)&\(5\)\tabularnewline[0pt]
\(L_{6,27}\)&\(3\)&\(\checkmark\)&\(32\)&\(13\)\tabularnewline[0pt]
\(L_{6,28}\)&\(2\)&\(\checkmark\)&\(8\)&\(3\)
\end{tabular}
}

\caption{Gradings of Lie algebras up to dimension 6 over \(\mathbb{C}\)}
\label{table-grading-info}
\end{table}

\begin{example}\label{example-gradings-of-heis-x-r}
We present our method of classifying all the possible gradings explicitly in the simple case of the Lie algebra \(L_{4,2}\) given in the basis \(Y_1,\ldots,Y_4\) with the only nonzero bracket \([Y_1,Y_2]=Y_3\). The {maximal grading} is over \(\mathbb{Z}^3\) with the {layers}
\begin{align*}
V_{(1,0,0)}\amp=\langle Y_1\rangle,\amp
V_{(0,1,0)}\amp=\langle Y_2\rangle,\amp
V_{(1,1,0)}\amp=\langle Y_3\rangle,\amp
V_{(0,0,1)}\amp=\langle Y_4\rangle\text{.}
\end{align*}

Ignoring scalar multiples, the difference set \(\Omega-\Omega\) of weights consists of the 6 elements \(e_1\), \(e_2\), \(e_1-e_2\), \(e_1-e_3\), \(e_2-e_3\), and \(e_1+e_2-e_3\), where \(e_1,e_2,e_3\) are the standard basis elements of the lattice \(\mathbb{Z}^3\). Subsets of these points span the trivial subspace, 6 one-dimensional subspaces, 7 two-dimensional subspaces \({\langle e_1,e_2\rangle}\), \({\langle e_1,e_3 \rangle}\), \({\langle e_2,e_3 \rangle}\), \({\langle e_1 - e_3,e_2 \rangle}\), \({\langle e_1,e_2 - e_3 \rangle}\), \({\langle e_1 - e_3,e_2 - e_3 \rangle}\), \({\langle 2e_1 - e_3,2e_2 - e_3 \rangle}\), and the full space \(\mathbb{Z}^3\).

In this case, each of these 15 subspaces \(S\) defines a torsion-free quotient \(\mathbb{Z}^3/S\). For instance parametrizing the quotient \(\pi\colon\mathbb{Z}^3\to\mathbb{Z}^3/\langle e_1 - e_3,e_2-e_3 \rangle\) as \(\mathbb{Z}\) using the complementary line \(\mathbb{Z}e_3\) gives the weights
\begin{equation*}
\pi(e_1) = \pi(e_2) = \pi(e_3) = 1,\quad \pi(e_1+e_2) = 2\text{,}
\end{equation*}
so a {push-forward grading} for the quotient \(\mathbb{Z}^3/\langle e_1 - e_3,e_2-e_3 \rangle\) is the {\(\mathbb{Z}\)-grading}
\begin{equation*}
V_1 = \langle Y_1, Y_2, Y_4\rangle, \quad V_2=\langle Y_3\rangle\text{.}
\end{equation*}

To determine the distinct {equivalence} classes out of the 15 gradings, we first consider the simple criteria listed earlier. The trivial grading and the {maximal grading} are distinguished by the rank. The six \(\mathbb{Z}^2\)-gradings all have 2 one-dimensional layers and 1 two-dimensional layer. There exists a homomorphism that preserves the dimensions of the layers for two pairs of the gradings: one between the quotients by \(\langle e_1 \rangle\) and \(\langle e_2\rangle\), and one between the quotients by \(\langle e_1-e_3\rangle \) and \(\langle e_2-e_3\rangle\).

Out of the seven \(\mathbb{Z}\)-gradings, the four quotients by
\begin{equation*}
{\langle e_1,e_2\rangle},
{\langle e_1 - e_3,e_2 \rangle},
{\langle e_1,e_2 - e_3 \rangle},
{\langle e_1 - e_3,e_2 - e_3 \rangle}
\end{equation*}
define gradings with 1 one-dimensional layer and 1 three-dimensional layer, and the three quotients by
\begin{equation*}
{\langle e_1,e_3 \rangle},
{\langle e_2,e_3 \rangle},
{\langle 2e_1 - e_3,2e_2 - e_3 \rangle}
\end{equation*}
define gradings with 2 two-dimensional layers. In both families there is exactly one pair of gradings admitting a homomorphism: the pair \(\langle e_1 - e_3,e_2 \rangle\) and \(\langle e_1,e_2 - e_3 \rangle\), and the pair \(\langle e_1,e_3 \rangle\) and \(\langle e_2,e_3 \rangle\).

In all of these cases, the homomorphism between the quotients is induced by the isomorphism \(f\colon\mathbb{Z}^3\to\mathbb{Z}^3\) swapping \(e_1\) and \(e_2\). All of the mentioned pairs of \(\mathbb{Z}^2\)- and \(\mathbb{Z}\)-gradings are in fact equivalent, since there is a corresponding Lie algebra automorphism swapping the basis elements \(Y_1\) and \(Y_2\) that preserves the subspaces \(\langle Y_3\rangle\) and \(\langle Y_4\rangle\). This reduces the list of 15 gradings down to 11 distinct equivalence classes. {Universal realizations} for each equivalence class of {torsion-free gradings} are listed in {Table~{\ref{table-grading-example}}}.
\begin{table}[hbtp]
\centering
{
\begin{tabular}[t]{lll}
rank&type&layers\tabularnewline\hline
3&(4)&\(V_{1,0,0}\oplus V_{0,1,0}\oplus V_{1,1,0}\oplus V_{0,0,1} = \langle Y_1\rangle\oplus \langle Y_2\rangle\oplus \langle Y_3\rangle\oplus \langle Y_4\rangle\)\tabularnewline[0pt]
2&(2, 1)&\(V_{0,0}\oplus V_{1,0}\oplus V_{0,1}  = \langle Y_2\rangle\oplus \langle Y_4\rangle\oplus \langle Y_1,Y_3\rangle\)\tabularnewline[0pt]
2&(2, 1)&\(V_{1,0}\oplus V_{0,1}\oplus V_{0,2} = \langle Y_4\rangle\oplus \langle Y_1,Y_2\rangle\oplus \langle Y_3\rangle\)\tabularnewline[0pt]
2&(2, 1)&\(V_{1,0}\oplus V_{0,1}\oplus V_{1,1} = \langle Y_1,Y_4\rangle\oplus \langle Y_2\rangle\oplus \langle Y_3\rangle\)\tabularnewline[0pt]
2&(2, 1)&\(V_{1,-1}\oplus V_{0,1}\oplus V_{1,0} = \langle Y_1\rangle\oplus \langle Y_2\rangle\oplus \langle Y_3,Y_4\rangle\)\tabularnewline[0pt]
1&(0, 2)&\(V_{0}\oplus V_{1} = \langle Y_1,Y_4\rangle \oplus \langle Y_2,Y_3\rangle\)\tabularnewline[0pt]
1&(0, 2)&\(V_{1}\oplus V_{2} = \langle Y_1,Y_2\rangle\oplus \langle Y_3,Y_4\rangle\)\tabularnewline[0pt]
1&(1, 0, 1)&\(V_{1}\oplus V_{2} = \langle Y_1,Y_2,Y_4\rangle\oplus \langle Y_3\rangle\)\tabularnewline[0pt]
1&(1, 0, 1)&\(V_{0}\oplus V_{1} = \langle Y_1,Y_2,Y_3\rangle \oplus \langle Y_4\rangle\)\tabularnewline[0pt]
1&(1, 0, 1)&\(V_{0}\oplus V_{1} = \langle Y_1\rangle\oplus \langle Y_2,Y_3,Y_4\rangle\)\tabularnewline[0pt]
0&(0, 0, 0, 1)&\(V_{0} = \langle Y_1,Y_2,Y_3,Y_4\rangle\)
\end{tabular}
}

\caption{Gradings of the Lie algebra \(L_{4,2}\)}
\label{table-grading-example}
\end{table}

\end{example}

\typeout{************************************************}
\typeout{Subsection 4.3 Enumerating Heintze groups}
\typeout{************************************************}

\subsection{Enumerating Heintze groups}\label{applications-heintze}
In this section, we present how knowing a maximal grading of a given nilpotent Lie algebra \(\mathfrak{g}\) can be used to determine a list of Heintze groups over \(\mathfrak{g}\). However, note that when working in the non-algebraically closed field \(\mathbb{R}\), we cannot in general obtain the maximal grading  using {Algorithm~{\ref{algorithm-maximal-grading}}}, see {Remark~{\ref{remark-non-algebraically-closed-fields}}}.
\begin{definition}\label{def-heintze}
A \terminology{Heintze group} is a simply connected Lie group over \(\mathbb{R}\) whose Lie algebra is a semidirect product of a nilpotent Lie algebra \(\mathfrak{g}\) and \(\mathbb{R}\) via a derivation \(\alpha \in \der(\mathfrak{g})\) whose eigenvalues have strictly positive real parts.
\end{definition}
Positive gradings for a given Lie algebra are naturally identified with diagonalizable derivations with strictly positive eigenvalues, see {Subsection~{\ref{sec-gradings-by-tori}}}. Hence, to any positively graded Lie algebra \(\mathfrak{g}\) we may associate a Heintze group over \(\mathfrak{g}\). We shall call these groups \terminology{diagonal Heintze groups}.

The quasi-isometric classification of Heintze groups reduces to the study of so called \terminology{purely real Heintze groups}, for which the associated derivation has real eigenvalues. Purely real Heintze groups are equivalent to diagonal Heintze groups under a slightly weaker notion of equivalence (sublinear biLipschitz-equivalence, see Theorem~ 1.2 of \cite{Cornulier-2019-sublinear_bilipschitz_equivalence} and Theorem~ 3.2 of \cite{pallier20-sublinear-qc}). Moreover, by \cite{Carrasco-Sequeira-2017-qi_invariants_associated_to_a_heintze_group} if two diagonal Heintze groups are quasi-isometric, then their associated derivations are proportional. Hence, the quasi-isometric classification problem of diagonal Heintze groups can be approached by treating the algebraic problem of finding all the possible derivations defining non-isomorphic diagonal Heintze groups.

{Proposition~{\ref{prop-noneqgradings-produce-diff-heintze}}} is a tool for tackling the above mentioned algebraic problem using positive gradings. We will prove this result later in this section after discussing its role in the enumeration of Heintze groups.
\begin{proposition}\label{prop-noneqgradings-produce-diff-heintze}
Let \(\mathfrak{g}\) be a nilpotent Lie algebra over \(\mathbb{R}\) and \(\alpha, \beta \in \der(\mathfrak{g})\) diagonalizable derivations with strictly positive eigenvalues. If \(\alpha\) and \(\beta\) define isomorphic Heintze groups, then they define {equivalent} \(\mathbb{R}\)-gradings.
\end{proposition}
The enumeration of positive gradings we have established immediately gives the corresponding enumeration of diagonal Heintze groups over \(\mathfrak{g}\). The enumeration of positive gradings can be understood in two different ways, as we discussed in {Subsection~{\ref{sssec-pos-gradings}}}. The corresponding enumeration of Heintze groups has similar character: it is either a parametrization via the projections or a finite list that does not contain all the isomorphism classes of Heintze groups but a representative for each family in terms of the layers.

Considering the parametrization of positive gradings via parametrization of the projections, we note that if one is able to eliminate equivalent gradings from the enumeration of positive gradings, then by {Proposition~{\ref{prop-noneqgradings-produce-diff-heintze}}} the corresponding list of Heintze groups does not contain isomorphic Heintze groups. Notice that already over \(\mathfrak{g}=\mathbb{R}^2\) there are uncountably many isomorphism classes of Heintze groups given by the projections \((1,0) \mapsto 1\) and \((0,1) \mapsto a\) with \(a >0\).

{Proposition~{\ref{prop-noneqgradings-produce-diff-heintze}}} will follow by a suitable conjugation of the derivations by the adjoint map. We will first recall some relevant formulas. We thank G.~Pallier for helping us improve an early version of {Lemma~{\ref{lemma-adjoint-conjugation-of-derivation}}}.

Recall that \(\Ad_{\exp(X)} = e^{\ad({X})}\), see \cite[Proposition~1.91]{knapp-beyond-introduction}, and recall the identity
\begin{equation*}
\Ad_{\exp(X)}\circ\ad({Y})\circ\Ad_{\exp(-X)} = \ad({\Ad_{\exp(X)}Y}).
\end{equation*}

\begin{lemma}\label{lemma-adjoint-conjugation-of-derivation}
Let \(\mathfrak{g}\) be a Lie algebra. Let \(\delta\in\der(\mathfrak{g})\) be a derivation and let \(X\in\mathfrak{g}\) be such that \([\delta(X),X]=0\). Then
\begin{equation*}
\Ad_{\exp(X)} \circ \delta \circ \Ad_{\exp(-X)} = \delta - \ad({\delta(X)})\text{.}
\end{equation*}

\end{lemma}
\begin{proof}\label{g:proof:idp24}
A consequence of the assumption \([\delta(X),X]=0\) is that if \(P\) is a polynomial and \(P'\) denotes its derivative polynomial, then we have
\begin{equation*}
[\delta,P(\ad(X))] = \ad({\delta(X)}) \circ P'(\ad(X))\text{.}
\end{equation*}
In the limit we obtain
\begin{equation*}
[\delta,e^{-\ad(X)}] = \ad({\delta(-X)}) \circ e^{-\ad(X)}.
\end{equation*}
Expanding out the bracket in \(e^{\ad(X)}[\delta,e^{-\ad(X)}]\) and reorganizing terms making use of the assumption \([\delta(X),X]=0\), we obtain the desired formula. 
\end{proof}
\begin{lemma}\label{proposition-derivation-plus-ad-is-conjugate}
Let \(\mathfrak{g}\) be a Lie algebra over \(\mathbb{R}\). Let \(\delta\in\der(\mathfrak{g})\) be a diagonalizable derivation with all eigenvalues strictly positive. Then for every vector \(Y\in\mathfrak{g}\) there exists a vector \(X\in\mathfrak{g}\) such that \(\Ad_{\exp(X)}\circ\delta\circ\Ad_{\exp(-X)}=\delta-\ad(Y)\).
\end{lemma}
\begin{proof}\label{g:proof:idp25}
For a vector \(X\in\mathfrak{g}\), denote by \(C_X\colon\der(\mathfrak{g})\to\der(\mathfrak{g})\) the conjugation map
\begin{equation*}
C_X(\eta) = \Ad_{\exp(X)}\circ \eta \circ \Ad_{\exp(-X)}\text{.}
\end{equation*}
Let \(X_1,\ldots,X_n\) be a basis of \(\mathfrak{g}\) that diagonalizes \(\delta\). Consider the map
\begin{equation*}
\Phi\colon \mathbb{R}^n\to \der(\mathfrak{g}),\quad \Phi(x_1,\ldots,x_n)=C_{x_nX_n}\circ\cdots\circ C_{x_1X_1}(\delta)\text{.}
\end{equation*}
By repeated application of {Lemma~{\ref{lemma-adjoint-conjugation-of-derivation}}}, it follows that \(\Phi(x) = \delta - \ad({\phi(x)})\), where \(\phi\colon \mathbb{R}^n\to\mathfrak{g}\) is the map
\begin{align}
\phi(x_1,\amp\ldots,x_n) = \delta(x_nX_n) + \Ad_{\exp(x_nX_n)}\delta(x_{n-1}X_{n-1})+\cdots\label{eq-def-adjoint-sum}\\
\amp+ \Ad_{\exp(x_nX_n)}\Ad_{\exp(x_{n-1}X_{n-1})}\cdots\Ad_{\exp(x_2X_2)}\delta(x_{1}X_{1})\text{.}\notag
\end{align}
Since the composition of conjugations is a conjugation, it suffices to prove that the map \(\phi\) is surjective.

Let \(w_1,\ldots,w_n\gt 0\) be the eigenvalues of the vectors \(X_1,\ldots,X_n\) for the derivation \(\delta\). Since the maps \(x_i\mapsto \operatorname{sign}(x_i)\abs{x_i}^{w_i}\) are all invertible, the map \(\phi\colon \mathbb{R}^n\to\mathfrak{g}\) is surjective if and only if the map \(\tilde{\phi}\colon \mathbb{R}^n\to\mathfrak{g}\) defined by
\begin{equation}
\tilde{\phi}(x_1,\ldots,x_n)=\phi(\operatorname{sign}(x_1)\abs{x_1}^{w_1},\ldots,\operatorname{sign}(x_n)\abs{x_n}^{w_n})\label{eq-rescaled-adjoint-sum}
\end{equation}
is surjective.

Let \(D_\lambda\in\Aut(\mathfrak{g})\), \(\lambda\gt 0\), be the one-parameter family of dilations defined by the derivation \(\delta\), i.e., \(D_\lambda = \exp(\delta \log \lambda)\). Then for each \(i=1,\ldots,n\) the dilation is given by \(D_\lambda(X_i)=\lambda^{w_i}X_i\) and we have the dilation equivariance
\begin{equation*}
\Ad_{\exp(\lambda^{w_i} X_i)}\circ D_\lambda = D_{\lambda}\circ\Ad_{\exp(X_i)}\text{.}
\end{equation*}

Applying the above equivariance to the definition {({\ref{eq-rescaled-adjoint-sum}})} we find that the map \(\tilde{\phi}\) is \(D_\lambda\)-homogeneous, i.e.\@, \(\tilde{\phi}(\lambda x) = D_\lambda(\tilde{\phi}(x))\) for all \(x\in\mathbb{R}^n\) and \(\lambda\gt 0\). Since \(\bigcup_{\lambda > 0} D_\lambda(U) = \mathfrak g\) for any neighborhood \(U\) of the identity it follows that the map \(\tilde{\phi}\) is surjective if and only if it is open at zero. Since the change of parameters in {({\ref{eq-rescaled-adjoint-sum}})} is a homeomorphism, the same is true also for the map \(\phi\).

By the definition {({\ref{eq-def-adjoint-sum}})}, the map \(\phi\) is smooth. The derivative of each summand \(\Ad_{\exp(x_nX_n)}\cdots\Ad_{\exp(x_{i+1}X_{i+1})}\delta(x_iX_i)\) at zero is the map \(x\mapsto \delta(x_iX_i)\), so the derivative \(D_0\phi\) of the map \(\phi\) at zero is
\begin{equation*}
D_0\phi(x_1,\ldots,x_n) = \delta(x_1X_1+\cdots+x_nX_n)\text{.}
\end{equation*}
By the strictly positive eigenvalue assumption, the map \(\delta\) is invertible. Since \(X_1,\ldots,X_n\) is a basis of \(\mathfrak{g}\), it follows that the map \(\phi\) is open at zero, concluding the proof.
\end{proof}
\begin{proof}[Proof of Proposition~{\ref*{prop-noneqgradings-produce-diff-heintze}}]\label{g:proof:idp26}
Rescaling the derivations by a scalar, we may assume that the smallest of the eigenvalues for both derivations is 1. Since the Heintze groups are assumed to be isomorphic, it is straightforward to see that there is a vector \(X \in \mathfrak{g}\) such that the derivation \(\alpha\) is conjugate by a Lie algebra automorphism of \(\mathfrak{g}\) to the derivation \(\beta + \ad(X)\). By {Lemma~{\ref{proposition-derivation-plus-ad-is-conjugate}}}, it follows that \(\alpha\) and \(\beta\) are conjugate. Applying {Lemma~{\ref{lemma-conjugates-and-subtori}}}\ref{enum-conj} to the {split tori} spanned by \(\alpha\) and \(\beta\) gives the desired result.
\end{proof}

\typeout{************************************************}
\typeout{Subsection 4.4 Bounds for non-vanishing \(\ell^{q,p}\) cohomology}
\typeout{************************************************}

\subsection{Bounds for non-vanishing \(\ell^{q,p}\) cohomology}\label{applications-lpq}
Knowing all the possible {positive gradings} of a nilpotent Lie algebra \(\mathfrak{g}\) has one further application in the realm of quasi-isometric classifications. The parametrization of all the possible positive gradings combined with the technical tools presented in \cite{pansu-rumin} can be used to find improved vanishing estimates for the \(\ell^{q,p}\) cohomology of a nilpotent Lie group, which is a well-known quasi-isometry invariant. In this section we present a systematic way of obtaining these estimates using the theory considered in the previous sections. Note that the methods of this paper for computing the maximal grading require the Lie algebra to be defined over an algebraically closed field, however see {Remark~{\ref{remark-non-algebraically-closed-fields}}}.

By definition, the \(\ell^{q,p}\) cohomology of a Riemannian manifold with bounded geometry is the \(\ell^{q,p}\) cohomology of every bounded geometry simplicial complex quasi-isometric to it. A crucial result of \cite{pansu-rumin} shows that in the case of contractible Lie groups, the \(\ell^{q,p}\) cohomology of the manifold is isomorphic to its {\(L^{q,p}\) cohomology}.
\begin{definition}\label{def-lpq}
The \terminology{\(L^{q,p}\) cohomology} of a nilpotent Lie group \(G\) is defined as
\begin{equation*}
L^{q,p}H^\bullet(G)=\frac{\lbrace \text{closed forms in }L^p\rbrace}{d\big(\lbrace \text{forms in }L^q\rbrace\big)\cap L^p}\text{.}
\end{equation*}

\end{definition}
In \cite[Theorem~1.1]{pansu-rumin}  it is shown that the Rumin complex constructed on a {Carnot group} allows for sharper computations regarding \(L^{q,p}H^\bullet(G)\) when compared to the usual de Rham complex. Defining and reviewing the properties of the Rumin complex \((E_0^\bullet,d_c)\) goes beyond the scope of this paper. For the following discussion, it is sufficient to know that the space of Rumin \(h\)-forms \(E_0^h\) is a subspace of the space of smooth differential \(h\)-forms of the underlying nilpotent Lie group \(G\).
\begin{definition}\label{def-weights-of-forms}
Let us consider a {positive grading} \(\mathcal{V} : \mathfrak{g}=\bigoplus_{\alpha\in\mathbb{R}}V_\alpha\). If \(\theta=X^\ast\) for \(X \in V_\alpha\), then we say that the left-invariant 1-form \(\theta\) has \terminology{weight} \(\alpha\) and denote \(w(\theta)=\alpha\). In general, given a left-invariant \(h\)-form, we say that it has \terminology{weight} \(p\) if it can be expressed as a linear combination of left-invariant \(h\)-forms \(\theta_{i_1,\ldots,i_h} = \theta_{i_1}\wedge\cdots\wedge\theta_{i_h}\) such that \(w(\theta_{i_1})+\cdots+w(\theta_{i_h})=p\).
\end{definition}
Given a positive grading \(\mathcal V :\mathfrak{g} = \bigoplus_{\alpha \in \mathbb{R}} V_\alpha\), we call the quantity
\begin{equation*}
Q = \sum_{\alpha \in \mathbb{R}_+}\alpha \dim V_\alpha
\end{equation*}
the \terminology{homogeneous dimension} of \(\mathcal V\). We also define for each degree \(h\) the number
\begin{gather*}
\delta N_{\min}(h)= \min_{\theta\in E_0^h} w(\theta)-\max_{\tilde{\theta}\in E_0^{h-1}}w(\tilde{\theta})\,.
\end{gather*}

The following is \cite[Theorem~1.1(ii)]{pansu-rumin}.
\begin{theorem}\label{thm-nonvanishing-lpq-bounds}
Let \(G\) be a {Carnot group} of homogeneous dimension \(Q\). If
\begin{equation*}
1\le p,q\le\infty\textit{ and } \frac{1}{p}-\frac{1}{q}\lt \frac{\delta N_{\min}(h)}{Q}
\end{equation*}
then the {\(L^{q,p}\) cohomology} of \(G\) in degree \(h\) does not vanish.
\end{theorem}
Moreover, in Theorem 9.2 of the same paper it is shown how the non-vanishing statement has a wider scope, as it can be applied to Carnot groups equipped with a homogeneous structure that comes from a {positive grading}. This result has been further extended in \cite{Tripaldi-2020-rumin_complex} to arbitrary positively graded nilpotent Lie groups.

A natural question that stems from these considerations is whether it is possible to identify which choice of positive grading will yield the best interval for non-vanishing cohomology. This problem can be easily presented in terms of maximising the value of the fraction \(\delta N_{\min}(h)/Q\) among all the possible positive gradings for a given Lie group \(G\).

Let us describe the maximization procedure in more detail. Recall that positive gradings of \(\mathfrak{g}\) are parametrized by vectors \(\mathbf{a}\in\positiveset\) where \(\positiveset\) is defined by {({\ref{eq-positive-set}})}. Denote by \(w(\theta)_{\mathbf{a}}\) the {weight} of a one-form \(\theta\) for the positive grading associated with a vector \(\mathbf{a}\in\positiveset\). Then we want to find the value of the following expression for each degree \(h\):
\begin{gather*}
\max_{\mathbf{a}\in\positiveset}\bigg\{\frac{\min_{\theta\in E_0^h}w(\theta)_\mathbf{a}-\max_{\tilde{\theta}\in E_0^{h-1}}w(\tilde{\theta})_{\mathbf{a}}}{Q_{\mathbf{a}}}\bigg\}\text{,}
\end{gather*}
where \(Q_{\mathbf{a}}\) is the homogeneous dimension of \(\pi^{\mathbf a}_*(\mathcal W)\).

\label{conversion-into-linear-optimizatio-problem}
A problem of this form can be converted into a linear optimization problem as follows:
\begin{enumerate}[label=\arabic*.]
\item{}replace \(\min_{\theta\in E_0^h} w(\theta)_\mathbf{a}\) with a new variable \(x\), and add the constraint \(x\le w(\theta)_\mathbf{a}\) for each \(\theta\in E_0^h\);
\item{}replace \(\max_{\tilde{\theta}\in E_0^{h-1}}w(\tilde{\theta})\) with a new variable \(y\), and add the constraint \(y\ge w(\tilde{\theta})_\mathbf{a}\) for each \(\tilde{\theta}\in E_0^{h-1}\);
\item{}normalize the expression by imposing \(Q_\mathbf{a}=1.\)
\end{enumerate}
We are then left with the following expression for our original maximization problem
\begin{align*}
\text{Maximize}\quad \amp x-y\\
\text{subject to}\quad x\amp\le w(\theta)_\mathbf{a}\quad \forall\,\theta\in E_0^h,\\
y\amp\ge w(\tilde{\theta})_\mathbf{a}\quad \forall \tilde{\theta}\in E_0^{h-1},\\
Q_\mathbf{a}\amp=1,\quad
\mathbf{a}\in\positiveset
\end{align*}
which can easily be solved by a computer, yielding the optimal bound for non-vanishing cohomology using the method of {Theorem~{\ref{thm-nonvanishing-lpq-bounds}}}.
\begin{example}\label{bounds-for-L-6-10-group}
Let us consider the non-stratifiable Lie group \(G\) of dimension 6, whose Lie algebra is denoted as \(L_{6,10}\) in \cite{deGraaf-2007-dim_6_nilpotent_lie_algebras}, with the non-trivial brackets
\begin{equation*}
[X_1,X_2]=X_3\;,\;[X_1,X_3]=[X_5,X_6]=X_4\text{.}
\end{equation*}

The space of Rumin forms in \(G\) is
\begin{align*}
E_0^1\amp=\langle\theta_1,\theta_2,\theta_5,\theta_6\rangle;\\
E_0^2\amp=\langle\theta_{5,6}-\theta_{1,3},\theta_{1,5},\theta_{1,6},\theta_{2,3},\theta_{2,5},\theta_{2,6}\rangle;\\
E_0^3\amp=\langle\theta_{2,5,6}+\theta_{1,2,3},\theta_{2,3,5},\theta_{2,3,6},\theta_{1,3,4}-\theta_{4,5,6},\theta_{1,4,5},\theta_{1,4,6}\rangle.
\end{align*}

For the Lie algebra \(L_{6,10}\), the {maximal grading} is over \(\mathbb{Z}^3\) with the {layers}
\begin{align*}
V_{(0,1,0)}\amp=\langle X_1\rangle,\amp
V_{(0,0,1)}\amp=\langle X_2\rangle,\amp
V_{(0,1,1)}\amp=\langle X_3\rangle\\
V_{(0,2,1)}\amp=\langle X_4\rangle,\amp
V_{(1,0,0)}\amp=\langle X_5\rangle,\amp
V_{(-1,2,1)}\amp=\langle X_6\rangle\text{.}
\end{align*}
The family of projections \(\pi^{\mathbf{a}} \colon \mathbb{Z}^3 \to \mathbb{R}\) giving positive gradings is parametrized by \((a_1,a_2,a_3)=\mathbf{a}\in\positiveset\) as in {({\ref{eq-positive-set}})}. The {weights} of left-invariant 1-forms are
\begin{align*}
w(\theta_1)_{\mathbf{a}}\amp=\pi^{\mathbf{a}}(0,1,0)=a_2;\\
w(\theta_2)_{\mathbf{a}}\amp=\pi^{\mathbf{a}}(0,0,1)=a_3;\\
w(\theta_3)_{\mathbf{a}}\amp=\pi^{\mathbf{a}}(0,1,1)=a_2+a_3;\\
w(\theta_4)_{\mathbf{a}}\amp=\pi^{\mathbf{a}}(0,2,1)=2a_2+a_3;\\
w(\theta_5)_{\mathbf{a}}\amp=\pi^{\mathbf{a}}(1,0,0)=a_1;\\
w(\theta_6)_{\mathbf{a}}\amp=\pi^{\mathbf{a}}(-1,2,1)=2a_2+a_3-a_1\text{.}
\end{align*}
From this computation we get the explicit expression
\begin{equation*}
\positiveset = \{\mathbf{a}\in\mathbb{R}^3: a_1\gt 0,\,a_2\gt 0,\,a_3\gt 0,\,-a_1+2a_2+a_3\gt 0 \}
\end{equation*}
and the homogeneous dimension \(Q_{\mathbf{a}}=6a_2+4a_3\).

Let us first consider the bound for non-vanishing cohomology in degree 1. We express
\begin{gather*}
\max_{\mathbf{a}\in\positiveset}\bigg\{\frac{\delta N_{\min}(1)}{Q_{\mathbf{a}}}\bigg\}=\max_{\mathbf{a}\in\positiveset}\bigg\{\frac{\min\{a_1,a_2,a_3,2a_2+a_3-a_1\}}{6a_2+4a_3}\bigg\}
\end{gather*}
as the linear optimization problem
\begin{align*}
\text{Maximize}\quad \amp x\\
\text{subject to}\quad x\amp\le a_1,\;
x\leq a_2,\;
x\leq a_3,\\
x\amp\leq 2a_2+a_3-a_1,\\
1\amp=6a_2+4a_3,\\
a_1\amp,a_2,a_3\gt 0,\; 2a_2+a_3-a_1\gt 0\text{.}
\end{align*}
A solver finds the solution \(\frac{1}{10}\), which is obtained by choosing \(a_1=a_2=a_3=\frac{1}{10}\). Since the quantity \(\frac{\delta N_{\min}(1)}{Q_{\mathbf{a}}}\) is scaling invariant, we find that the grading defined by \(a_1=a_2=a_3=1\) gives \(\ell^{q,p}H^1(G)\neq 0\) with the optimal bound \(\frac{1}{p}-\frac{1}{q}\lt \frac{1}{10}\).

Similarly, once we re-express
\begin{gather*}
\max_{\mathbf{a}\in\positiveset}\bigg\{\frac{\delta N_{\min}(2)}{Q_{\mathbf{a}}}\bigg\} 
\end{gather*}
as a linear optimization problem and feed it into a solver, we get the result \(\frac{1}{10}\), obtained (up to rescaling) by taking \(a_2=a_3=2\) and \(a_1=3\). Therefore \(\ell^{q,p}H^2(G)\neq 0\) for \(\frac{1}{p}-\frac{1}{q}\lt\frac{1}{10}\).

Likewise, we obtain the optimal bound \(\frac{1}{p}-\frac{1}{q}\lt\frac{1}{10}\) for \(\ell^{q,p}H^3(G)\neq 0\) by taking \(a_1=a_2=a_3=1\).

Finally, by Hodge duality, see \cite[Theorem~7.3]{Tripaldi-2020-rumin_complex}, we obtain the optimal bounds for \(\ell^{q,p}\) cohomology in complementary degree, that is \(\ell^{q,p}H^4(G)\neq 0\), \(\ell^{q,p}H^5(G)\neq 0\), and \(\ell^{q,p}H^6(G)\neq 0\), for \(\frac{1}{p}-\frac{1}{q}\lt\frac{1}{10}\).
\end{example}
\begin{remark}\label{remark-Engel-example}
\cite[Example~9.5]{pansu-rumin} describes an explicit {positive grading} in the Engel group that gives an improved bound for the non-vanishing of the {\(L^{q,p}\) cohomology} in degree 2. By a similar computation as the one shown in {Example~{\ref{bounds-for-L-6-10-group}}}, one can verify that the value given in \cite[Example~9.5]{pansu-rumin} is indeed the optimal bound.
\end{remark}

\appendix

\addtocontents{toc}{\vspace{\normalbaselineskip}}

\typeout{************************************************}

\textbf{Acknowledgements.}
All of the authors were supported by the Academy of Finland (grant 288501 Geometry of subRiemannian groups and by grant 322898 Sub-Riemannian Geometry via Metric-geometry and Lie-group Theory) and by the European Research Council (ERC Starting Grant 713998 GeoMeG Geometry of Metric Groups).
E.H. was also supported by the Vilho, Yrjö and Kalle Väisälä Foundation, and by the SISSA project DIP\_ECC\_MATE\_CoordAreaMate\_0459 - Dipartimenti di Eccellenza 2018 - 2022 (CUP: G91|18000050006).
V.K. was also supported by the Emil Aaltonen foundation.
F.T. was also supported by the University of Bologna, funds for selected research topics, and by the European Union's Horizon 2020 research and innovation programme under the Marie Skłodowska-Curie grant agreement No 777822 GHAIA Geometric and Harmonic Analysis with Interdisciplinary Applications, and the Swiss National Foundation grant 200020\_191978.

\bibliographystyle{amsalpha}
\bibliography{ref}

\newcommand{\etalchar}[1]{$^{#1}$}
\providecommand{\bysame}{\leavevmode\hbox to3em{\hrulefill}\thinspace}
\providecommand{\MR}{\relax\ifhmode\unskip\space\fi MR }
\providecommand{\MRhref}[2]{%
  \href{http://www.ams.org/mathscinet-getitem?mr=#1}{#2}
}
\providecommand{\href}[2]{#2}
\begin{thebibliography}{CdGSGT18}

\bibitem[CdGS12]{Cicalo-deGraaf-Schneider-2012-6d_nilpotent_lie_algebras}
Serena Cical\`{o}, Willem~A. de~Graaf, and Csaba Schneider,
  \emph{Six-dimensional nilpotent {L}ie algebras}, Linear Algebra Appl.
  \textbf{436} (2012), no.~1, 163--189. \MR{2859920}

\bibitem[CdGSGT18]{LieAlgDB2.2}
Serena Cical\`{o}, Willem~A. de~Graaf, Csaba Schneider, and The GAP~Team,
  \emph{{LieAlgDB}, a database of lie algebras, {V}ersion 2.2},
  \href{https://gap-packages.github.io/liealgdb/}
  {\texttt{https://gap-packages.github.io/}\discretionary
  {}{}{}\texttt{liealgdb/}}, Apr 2018, Refereed GAP package.

\bibitem[CKLD{\etalchar{+}}17]{CKLNO-2017-homogeneous_metric_spaces_to_lie_groups}
Michael~G. Cowling, Ville Kivioja, Enrico Le~Donne, Sebastiano Nicolussi~Golo,
  and Alessandro Ottazzi, \emph{From homogeneous metric spaces to {L}ie
  groups}, arXiv e-prints (2017), arXiv:1705.09648.

\bibitem[Cor16]{Cornulier-2016-gradings}
Yves Cornulier, \emph{Gradings on {L}ie algebras, systolic growth, and
  cohopfian properties of nilpotent groups}, Bull. Soc. Math. France
  \textbf{144} (2016), no.~4, 693--744. \MR{3562610}

\bibitem[Cor18]{quasisurvey}
Yves~de Cornulier, \emph{On the quasi-isometric classification of locally
  compact groups}, New directions in locally compact groups, London Math. Soc.
  Lecture Note Ser., vol. 447, Cambridge Univ. Press, Cambridge, 2018,
  pp.~275--342. \MR{3793294}

\bibitem[Cor19]{Cornulier-2019-sublinear_bilipschitz_equivalence}
Yves Cornulier, \emph{On sublinear bilipschitz equivalence of groups}, Ann.
  {E}{N}{S} \textbf{52} (2019), no.~5, 1201--1242.

\bibitem[CPS17]{Carrasco-Sequeira-2017-qi_invariants_associated_to_a_heintze_group}
Matias Carrasco~Piaggio and Emiliano Sequeira, \emph{On quasi-isometry
  invariants associated to a {H}eintze group}, Geom. Dedicata \textbf{189}
  (2017), 1--16. \MR{3667336}

\bibitem[CR19]{conti-rossi-2019-nice_lie_groups}
Diego Conti and Federico~A. Rossi, \emph{Construction of nice nilpotent {L}ie
  groups}, J. Algebra \textbf{525} (2019), 311--340. \MR{3911646}

\bibitem[dG00]{deGraaf-2000-Lie_algebras_theory_and_algorithms}
Willem~A. de~Graaf, \emph{Lie algebras: theory and algorithms}, North-Holland
  Mathematical Library, vol.~56, North-Holland Publishing Co., Amsterdam, 2000.
  \MR{1743970}

\bibitem[dG07]{deGraaf-2007-dim_6_nilpotent_lie_algebras}
\bysame, \emph{Classification of 6-dimensional nilpotent {L}ie algebras over
  fields of characteristic not 2}, J. Algebra \textbf{309} (2007), no.~2,
  640--653. \MR{2303198}

\bibitem[EK13]{Elduque-Kochetov-2013-gradings_on_simple_lie_algebras}
Alberto Elduque and Mikhail Kochetov, \emph{Gradings on simple {L}ie algebras},
  Mathematical Surveys and Monographs, vol. 189, American Mathematical Society,
  Providence, RI; Atlantic Association for Research in the Mathematical
  Sciences (AARMS), Halifax, NS, 2013. \MR{3087174}

\bibitem[Eld10]{Elduque-2010-fine_gradings_on_simple_lie_algebras}
Alberto Elduque, \emph{Fine gradings on simple classical {L}ie algebras}, J.
  Algebra \textbf{324} (2010), no.~12, 3532--3571. \MR{2735398}

\bibitem[Fav73]{Favre-1973-system_de_poids}
Gabriel Favre, \emph{Syst\`eme de poids sur une alg\`ebre de {L}ie nilpotente},
  Manuscripta Math. \textbf{9} (1973), 53--90. \MR{349780}

\bibitem[Gon98]{gong}
Ming-Peng Gong, \emph{Classification of nilpotent {L}ie algebras of dimension 7
  (over algebraically closed fields and {R})}, ProQuest LLC, Ann Arbor, MI,
  1998, Thesis (Ph.D.)--University of Waterloo (Canada). \MR{2698220}

\bibitem[Hei74]{Heintze-1974-homogeneous_manifolds_of_negative_curvature}
Ernst Heintze, \emph{On homogeneous manifolds of negative curvature}, Math.
  Ann. \textbf{211} (1974), 23--34. \MR{353210}

\bibitem[HKMT21]{software-lie_algebra_gradings}
Eero Hakavuori, Ville Kivioja, Terhi Moisala, and Francesca Tripaldi,
  \emph{ehaka/lie-algebra-gradings: v1.1}, June 2021,
  https://doi.org/10.5281/zenodo.5040282.

\bibitem[Hum78]{Humphreys-1978-representation_theory}
James~E. Humphreys, \emph{Introduction to {L}ie algebras and representation
  theory}, Graduate Texts in Mathematics, vol.~9, Springer-Verlag, New
  York-Berlin, 1978, Second printing, revised. \MR{499562}

\bibitem[Kna02]{knapp-beyond-introduction}
Anthony~W. Knapp, \emph{Lie groups beyond an introduction}, second ed.,
  Progress in Mathematics, vol. 140, Birkh\"{a}user Boston, Inc., Boston, MA,
  2002. \MR{1920389}

\bibitem[Koc09]{Kochetov-2009-gradings_on_simple_lie_algebras_survey}
Mikhail Kochetov, \emph{Gradings on finite-dimensional simple {L}ie algebras},
  Acta Appl. Math. \textbf{108} (2009), no.~1, 101--127. \MR{2540960}

\bibitem[LD17]{Le_Donne-2017-carnot_primer}
Enrico Le~Donne, \emph{A primer on {C}arnot groups: homogenous groups,
  {C}arnot-{C}arath\'{e}odory spaces, and regularity of their isometries},
  Anal. Geom. Metr. Spaces \textbf{5} (2017), no.~1, 116--137. \MR{3742567}

\bibitem[Mag08]{magnin-huge-book}
Louis Magnin, \emph{Adjoint and trivial cohomologies of nilpotent complex {L}ie
  algebras of dimension {$\leq 7$}}, Int. J. Math. Math. Sci. (2008), Art. ID
  805305, 12. \MR{2461422}

\bibitem[Pal20]{pallier20-sublinear-qc}
Gabriel Pallier, \emph{Sublinear quasiconformality and the large-scale geometry
  of {H}eintze groups}, Conform. Geom. Dyn. \textbf{24} (2020), 131--163.
  \MR{4127909}

\bibitem[PR18]{pansu-rumin}
Pierre Pansu and Michel Rumin, \emph{On the {$\ell^{q,p}$} cohomology of
  {C}arnot groups}, Ann. H. Lebesgue \textbf{1} (2018), 267--295. \MR{3963292}

\bibitem[PZ89]{Patera-Zassenhaus-1989_on_lie_gradings_i}
Jiří Patera and Hans Zassenhaus, \emph{On {L}ie gradings. {I}}, Linear
  Algebra Appl. \textbf{112} (1989), 87--159. \MR{976333}

\bibitem[Sie86]{Siebert-1986-contractive_automorphisms}
Eberhard Siebert, \emph{Contractive automorphisms on locally compact groups},
  Math. Z. \textbf{191} (1986), no.~1, 73--90. \MR{812604}

\bibitem[Spr09]{springer-book}
Tonny~A. Springer, \emph{Linear algebraic groups}, second ed., Modern
  Birkh\"{a}user Classics, Birkh\"{a}user Boston, Inc., Boston, MA, 2009.
  \MR{2458469}

\bibitem[Tri20]{Tripaldi-2020-rumin_complex}
Francesca Tripaldi, \emph{The {R}umin complex on nilpotent {L}ie groups}, arXiv
  e-prints (2020), arXiv:2009.10154.

\end{thebibliography}

\end{document}